\documentclass[11pt,reqno]{amsart}
\usepackage{microtype}
\usepackage[margin=3cm]{geometry}

\usepackage[shortlabels]{enumitem}
\setlist[enumerate]{label={(\arabic*)}}

\usepackage{amssymb}
\usepackage[dvipsnames]{xcolor}
\usepackage
[colorlinks=true,linkcolor=Maroon,citecolor=OliveGreen,backref]
{hyperref}
\usepackage[abbrev, alphabetic]{amsrefs}  % bibliography
\usepackage[capitalize]{cleveref} %, autonum} \renewcommand{\qedhere}{\GenericError{}{qedhere plays up with autonum}{}{}}
\crefname{equation}{}{}

\usepackage[norefs, nocites]{refcheck}
\newcommand\phantomref[1]{\setbox0=\hbox{\ref{#1}}}

\makeatletter
\newcommand{\refcheckize}[1]{%
  \expandafter\let\csname @@\string#1\endcsname#1%
  \expandafter\DeclareRobustCommand\csname relax\string#1\endcsname[1]{%
    \csname @@\string#1\endcsname{##1}\wrtusdrf{##1}}%
  \expandafter\let\expandafter#1\csname relax\string#1\endcsname
}
\makeatother
\refcheckize{\cref}
\refcheckize{\Cref}

\numberwithin{equation}{section}

\newtheorem{lemma}{Lemma}[section]
\newtheorem{theorem}[lemma]{Theorem}
\newtheorem{proposition}[lemma]{Proposition}
\newtheorem{corollary}[lemma]{Corollary}
\newtheorem{fact}[lemma]{Fact}
\newtheorem{convention}[lemma]{Convention}

\theoremstyle{definition}
\newtheorem{remark}[lemma]{Remark}

\newcommand\opr[1]{\operatorname{#1}}

\newcommand{\eps}{\epsilon}
\newcommand{\vareps}{\varepsilon}

\def\R{\mathbf{R}}

\def\C{\mathrm{C}}
\def\Zint{\mathbf{Z}}
\def\F{\mathbf{F}}

\def\Z{\mathrm{Z}}

\def\Prob{\mathbf{P}}

\def\P{\mathcal{P}}
\def\calQ{\mathcal{Q}}
\def\I{\mathcal{I}}

\def\norm{\mathcal{N}}

\newcommand\br[1]{{\left(#1\right)}}
\newcommand\floor[1]{\left\lfloor{#1}\right\rfloor}
\newcommand\ceil[1]{\left\lceil{#1}\right\rceil}

\newcommand{\gen}[1]{\langle{#1}\rangle}

\newcommand{\ind}[1]{\left[#1\right]}

\newcommand{\GL}{\opr{GL}}
\newcommand{\SL}{\opr{SL}}

\newcommand{\SO}{\opr{SO}}
\newcommand{\SU}{\opr{SU}}
\newcommand{\GU}{\opr{GU}}
\newcommand{\Or}{\opr{O}}
\newcommand{\PGL}{\opr{PGL}}

\newcommand{\Sp}{\opr{Sp}}
\newcommand{\POm}{\opr{P\Omega}}
\newcommand{\PSL}{\opr{PSL}}
\newcommand{\PSU}{\opr{PSU}}
\newcommand{\PSp}{\opr{PSp}}

\newcommand{\PGamL}{\opr{P\Gamma L}}

\newcommand{\Soc}{\opr{Soc}}
\newcommand{\Hom}{\opr{Hom}}
\newcommand{\Aut}{\opr{Aut}}

\def\ca{\mathcal}

\def\Inndiag{\opr{Inndiag}}
\def\sm{\smallsetminus}

\def\derang{\delta_{\textup{cc}}}
\def\derangss{\delta_{\textup{cc,ss}}}

\def\dim{\opr{dim}}

\def\nsgp{\trianglelefteq}

\usepackage{todonotes}

\begin{document}

\title[Conjugacy classes of derangements]{Conjugacy classes of derangements in finite groups of Lie type}

\author{Sean Eberhard}
\address{Sean Eberhard, Mathematical Sciences Research Centre, Queen's University Belfast, Belfast BT7~1NN, UK}
\email{s.eberhard@qub.ac.uk}

\author{Daniele Garzoni}
\address{Daniele Garzoni, Department of Mathematics, University of Southern California, Los Angeles, CA 90089-2532, USA}
\email{garzoni@usc.edu}

\thanks{SE is supported by the Royal Society. DG has been partially supported by a grant of the Israel Science Foundation No.
702/19, and has received funding from the European Research Council (ERC) under the European
Union’s Horizon 2020 research and innovation programme (grant agreement No. 850956).}

\begin{abstract}
    Let $G$ be a finite almost simple group of Lie type acting faithfully and primitively on a set $\Omega$. We prove an analogue of the Boston--Shalev conjecture for conjugacy classes: the proportion of conjugacy classes of $G$ consisting of derangements is bounded away from zero. This answers a question of Guralnick and Zalesski.
    The proof is based on results on the anatomy of palindromic polynomials over finite fields (with either reflective symmetry or conjugate-reflective symmetry).
\end{abstract}
\maketitle

\setcounter{tocdepth}{1}
\tableofcontents

\section{Introduction}

\subsection{Boston--Shalev for conjugacy classes}

Let $G$ be a finite group acting transitively on a set $\Omega$. An element $g\in G$ is called a \emph{derangement} if it acts without fixed points on $\Omega$. The study of derangements has a long history, going back to the origins of permutation group theory in the 19th century.

% Derangements have many connections with other parts
% of mathematics: see, for instance, Serre \cite{serre2003theorem} for some connections with topology and number theory.
% and Fein–Kantor–Schacher \cite{fein1981relative} for an application to the structure of the relative Brauer group of extensions of global fields.

An elementary lemma of Jordan asserts that, if $G$ is finite and $|\Omega| \geq 2$, then $G$ contains a derangement. This lemma has nice applications to topology and number theory; see, for instance, Serre \cite{serre2003theorem}.

Given this result, it is natural to ask whether transitive permutation groups must contain many derangements. Let $\delta(G,\Omega)$ be the proportion of derangements of $G$ on $\Omega$. Cameron--Cohen \cite{cameron_cohen} showed that $\delta(G, \Omega) \geq 1/|\Omega|$, and that this bound is attained if and only if $G$ is a $2$-transitive Frobenius group.

In many cases, one can obtain much stronger bounds. One of the motivations for this paper is the following theorem of {\L}uczak--Pyber  \cite{luczak1993random} (for alternating groups) and Fulman--Guralnick~\cites{fulman2003derangements, fulman2012_conjugacy_classes, fulman_guralnick_subspace, fulman2018extension_field} (for groups of Lie type), confirming a conjecture posed independently by Boston and Shalev.

\begin{theorem} ({\L}uczak--Pyber, Fulman--Guralnick)
\label{t_boston_shalev}
Let $G$ be a finite simple group acting transitively on a set $\Omega$ with $|\Omega|\geq 2$. Then $\delta(G,\Omega)\geq \delta$ for an absolute constant $\delta>0$.
\end{theorem}

Since the property of being a derangement is closed under conjugation,
it is also natural to consider the proportion of \emph{conjugacy classes} consisting of derangements or, equivalently, the probability that a uniformly random conjugacy class contains derangements.
We denote by $\derang(G, \Omega)$ the proportion of conjugacy classes of $G$ whose members act as derangements on $\Omega$.
In this paper we answer a question of Guralnick and Zalesski by confirming that the analogue of the Boston--Shalev conjecture holds for conjugacy classes in non-alternating finite simple groups (see \cite{guralnick2016conjugacy}*{p.~121}).

\begin{theorem}
    \label{thm:main-simple-case}
    Let $G$ be a finite simple group of Lie type acting transitively on a set $\Omega$ with $|\Omega| \ge 2$.
    Then $\derang(G, \Omega) \ge \eps$ for an absolute constant $\eps > 0$.
\end{theorem}

It is easy to see that the exclusion of alternating groups is necessary.
The conjugacy classes of $S_n$ are parameterized bijectively by partitions of $n$, and a uniformly random partition of $n$ has $n^{1/2-o(1)}$ singletons with high probability.
Hence the elements of (asymptotically) almost every conjugacy class of $S_n$ have many fixed points in any of its low-degree permutation representations.
Since approximately half of the conjugacy classes of $S_n$ are contained in $A_n$, the same follows for $A_n$.
See also \cite{erdos_szalay} for a striking stronger result.

We point out that \Cref{thm:main-simple-case} follows from Fulman--Guralnick \cite{fulman2012_conjugacy_classes} if $G$ has bounded rank, and if $G$ has large rank provided the point stabilizer is not a subspace stabilizer, a stabilizer of a direct sum decomposition, or an extension field subgroup. These cases, which we address in this paper, are the key ones, and new methods and ideas are required. See \Cref{subsec:relation,subsec:anatomy} for more details.

More generally we consider the case of almost simple groups acting primitively.
Recall that a group $G$ is \emph{almost simple} if there is a nonabelian simple group $S$ such that $S \le G \le \Aut(S)$.
It is known that \Cref{t_boston_shalev} fails for almost simple groups in general: the proportion of derangements can be roughly $1/\log |\Omega|$ (see \cite{fulman2003derangements}).
However, it turns out that for the conjugacy class weighting we can extend to the almost simple case.

\begin{theorem}
    \label{t_main}
    Let $G$ be a finite almost simple group of Lie type acting faithfully and primitively on a set $\Omega$.
    Then $\derang (G,\Omega) \geq \eps$ for an absolute constant $\eps>0$.
\end{theorem}

In the course of the proof we will also see that \Cref{t_main} holds for groups $G$ in the intervals $\SL_n(q) \le G \le \GL_n(q)$ and $\SU_n(q) \le G \le \GU_n(q)$, provided that $G'$ acts nontrivially.

To extend \Cref{thm:main-simple-case} to the full almost simple case, the main extra ingredient needed is an estimate of independent interest for the number of conjugacy classes of $G$ when $G$ contains field automorphisms of $S$. Denoting by $k(-)$ the number of conjugacy classes, we prove that\footnote{Here $X \asymp Y$ means there are implicit constants $c_1, c_2 > 0$ such that $c_1 X \le Y \le c_2 X$. Later we will similarly use $X \ll Y$ to mean $X \le C Y$ for some constant $C$, or in other symbols~$X = O(Y)$. This is standard notation in analytic number theory.}

\[
    k(G) \asymp \frac{k(N)}{|G : N|},
\]
where $N = G \cap \Inndiag(S)$. For example, $k(\PGamL_n(p^f))\asymp k(\PGL_n(p^f))/f$. Since by \cite{fulman2012_conjugacy_classes} we know the value of $k(N)$ (up to a multiplicative absolute constant), this determines the value of $k(G)$ (up to a multiplicative absolute constant) for any almost simple group of Lie type $G$. See \Cref{thm:conj_classes_almost_simple} for the precise formulation of the estimate, which also gives the asymptotic of $k(G)$ as the size of the field of definition tends to infinity.

\subsection{Relation to Boston--Shalev for elements}
\label{subsec:relation}

In general, the distribution on a finite group $G$ defined by the uniform distribution on conjugacy classes can be drastically different from the uniform distribution on elements.
We saw this already in the case of the symmetric group: while the number of fixed points of a random element of $S_n$ is approximately Poisson-distributed with mean $1$, a random conjugacy class of $S_n$ is associated to a random partition of $n$, which typically has around $n^{1/2}$ singletons.

Nevertheless, the connection between the two distributions is closer in the case of a finite simple group of Lie type, at least if we restrict to regular semisimple elements. For groups of bounded rank $r$ and level $q$, every conjugacy class consisting of regular semisimple elements satisfies
\begin{equation}
    \label{eq:cc-to-elements}
    c_r |G| / q^r < |g^G| < C_r |G| / q^r
\end{equation}
for constants $c_r, C_r > 0$. (This holds since $\C_G(g)$ is an extension of a maximal torus by a group of $r$-bounded order; see \cite[Theorem 14.2 and Proposition 25.2]{malle_testerman}.)
This simple observation enables us to deduce the conjugacy class result directly from the element result in the bounded-rank case.

As the rank grows the bounds \eqref{eq:cc-to-elements} deteriorate, and it is no longer trivial to compare the conjugacy class distribution to the elements distribution. However, formulas for the size of $|\C_G(g)|$ (see \cite[Proposition 25.2]{malle_testerman} again) show that a proportion $1-\epsilon$ of regular semisimple conjugacy classes have order at least $c_\eps |L| / q^r$, where $L$ is the universal covering  group of $G$.
%(Note, e.g., that this is immediate for $G=\PSL_n(q)$.)
Since in the proof of \Cref{thm:main-simple-case} we will count regular semisimple classes, and since $k(G)\asymp q^r/|L:G|$, we can deduce the original Boston--Shalev conjecture (\Cref{t_boston_shalev}) for non-alternating groups from \Cref{thm:main-simple-case}.

For the reverse deduction, one would have to show that a  proportion $1-\epsilon$ of the regular semisimple elements of $G$ are contained in conjugacy classes of size bounded by $C_\epsilon |L| / q^r$.
This statement is subtler than the previous one, but still it is likely true.

In fact we use a different and more direct approach. Bounds by Fulman--Guralnick \cite{fulman2012_conjugacy_classes} allow us to restrict to semisimple conjugacy classes. Then, we use the well-known correspondence between semisimple classes and polynomials to translate all the relevant questions into analogous questions about polynomials, which we study using function-field analytic number theory.
The results we establish for polynomials are of independent interest and described in the next section.

The resulting proof has some advantages. Notably, our method is uniform in $q$ (small $q$ do not receive special treatment), and in large rank it does not use any estimates for the proportion of regular semisimple elements. This is a key feature of our proof, which makes it suitable for certain applications, where the use of the aforementioned estimates is problematic. In addition we get the following result:

\begin{theorem}
	\label{t:semisimple_asymptotic}
    For every $\eps>0$, there exist $f_i(\eps)$, $i=1, \dots, 3$ such that the following holds. Assume that $G$ is a finite simple group of Lie type of rank $r$, acting primitively on a set $\Omega$ with $|\Omega|\geq 2$, and assume that the proportion of semisimple conjugacy classes in $G$ which contain derangements is bounded by $1-\eps$. Then, for $\alpha \in \Omega$, one of the following holds:
    \begin{enumerate}[(i)]
        \item $G$ is classical and $G_\alpha$ is the stabilizer of a subspace of dimension or codimension at most $f_1(\eps)$;
        \item $(G,G_\alpha) = (\Sp_{2r}(q), \SO^\pm_{2r}(q))$ with $q$ even;
        \item $r\leq f_2(\eps)$ and $G_\alpha$ is a maximal subgroup of maximal rank;
        \item $|G|\leq f_3(\eps)$.
    \end{enumerate}
\end{theorem}

In words, if the proportion of semisimple classes of $G$ containing derangements is bounded away from $1$, then either $G$ is a classical group and $G_\alpha$ is the stabilizer of a subspace of bounded dimension or codimension, or $(G,G_\alpha) = (\Sp_{2r}(q), \SO^\pm_{2r}(q))$, or $G$ has bounded rank and $G_\alpha$ has maximal rank. The reverse implication also holds. See \Cref{r:tends_to_one,t_c2,t_c3,thm:main-quasisimple-classes-4+} for quantitative bounds on the proportion of semisimple classes containing derangements.

We recall that maximal subgroups of maximal rank are maximal subgroups containing a maximal torus. For classical groups, these are subgroups of Aschbacher's classes $\ca C_1$, $\ca C_2$, and $\ca C_3$ (see, e.g., \cite{garzoni_mckemmie}*{Theorems~5.2 and 5.5} for a precise statement). For exceptional groups, see \cite{liebeck_saxl_seitz} for a classification.

\subsection{Anatomy of palindromic polynomials}
\label{subsec:anatomy}
The most important cases of \Cref{t_main} turn out to be closely connected with some results on the anatomy of palindromic polynomials over finite fields.
Here we call a polynomial \emph{palindromic} if its coefficients satisfy either reflective symmetry or conjugate-reflective symmetry.

Loosely speaking, the subject of \textit{anatomy} studies how the building blocks of a mathematical object are assembled together, particularly when the object is chosen at random. The building blocks of an integer are the prime divisors, so the anatomy of integers is concerned with the study of divisors of a random integer. Analogously, the anatomy of permutations is concerned with sets fixed by a random permutation, and the anatomy of polynomials with divisors of a random polynomial. We refer, for instance, to \cites{ford, ford--toolkit, EFG} for motivation and results in this rich subject.

Let us now focus on polynomials over finite fields. Let $\P$ be the set of monic polynomials over $\F_q$ and let $\P(n)$ be the set of those of degree $n$.
%  A basic anatomical quantity is the number $\pi_q(n)$ of irreducible elements of $\P(n)$, and it is well-known (see for example \cite{rosen-book}*{Proposition~2.1}) that
% % \begin{equation}
%     \label{eq:PNT}
%     \pi_q(n) = \frac1n \sum_{d \mid n} \mu(d) q^{n/d} = \frac1n \br{q^n - O(q^{n/2})}.
% \end{equation}
An interesting anatomical quantity is $H(n,k)$, the number of polynomials $f \in \P(n)$ having a divisor of degree $k$.
It was proved by Meisner~\cite{meisner}*{Theorem~1.2} (using methods related to those of \cites{ford,EFG}, where the analogues for integers and permutations were proved) that
\begin{equation}
    \label{eq:Hnk}
    H(n, k) \asymp \frac{q^n}{k^\delta (1 + \log k)^{3/2}}
\end{equation}
uniformly for $1 \le k \le n/2$, where
\[
    \delta = 1 - \frac{1 + \log \log 2}{\log 2} = 0.086\dots.
\]

Now, given a monic polynomial $f(X) = X^n + a_{n-1} X^{n-1} + \cdots + a_1 X + a_0$ with coefficients in $\F_q$ and $a_0 \ne 0$ we define
\[
    f^*(X) = X^n f(1/X) / f(0) = X^n + (a_1/a_0) X^{n-1} + \cdots + (a_{n-1}/a_0) X + (1/a_0).
\]
We say $f$ is \emph{$*$-symmetric} if $f^*=f$.
Similarly, if $q$ is a square we denote by $x \mapsto \bar x = x^{q^{1/2}}$ the involutory automorphism and we define
\[
    f^\dagger(X) = X^n \bar f(1/X) / \bar f(0) = X^n + (\bar a_1 / \bar a_0) X^{n-1} + \cdots + (\bar a_{n-1} / \bar a_0) X + (1 / \bar a_0).
\]
We say $f$ is \emph{$\dagger$-symmetric} if $f^\dagger = f$.
We use the term \emph{palindromic} generally to refer to either of these cases (this is not a standard usage, but it is convenient for us).

We remark that $*$-symmetric polynomials are closely related to ordinary polynomials. Indeed, each $*$-symmetric polynomial degree $2n$ with constant coefficient $1$ can be written (uniquely) in the form
\begin{equation}
    \label{eq:star-symmetric-reduction}
    f(X) = X^n g(X+1/X),
\end{equation}
with $g\in \mathcal P(n)$.

Crucial to the proof of \Cref{t_main} are several results in the anatomy of palindromic polynomials.
The key property that enables us to prove results about the anatomy of generic palindromic polynomials
is the multiplicative rule $(fg)^* = f^*g^*$, which holds provided $f(0), g(0) \ne 0$.
(Note however that we do not have additivity $(f+g)^* = f^* + g^*$.)
Similarly  $(fg)^\dagger = f^\dagger g^\dagger$ provied $f(0), g(0) \ne 0$.

Let $H^*(n, k)$ be the number of $*$-symmetric $f \in \P(n)$ having a $*$-symmetric factor of degree $k$, and define $H^\dagger(n, k)$ similarly. We will need, for example, to bound $H^*(n,k)$ and $H^\dagger(n,k)$; to bound the number of palindromic polynomials factorizing nearly as $gg^*$ or $gg^\dagger$; to show that approximately half of the $*$-symmetric polynomials of degree $n$ have an even number of irreducible factors. Thanks to \eqref{eq:Hnk} and \eqref{eq:star-symmetric-reduction}, the first task is immediate for $*$-symmetric polynomials:
\[
    H^*(n, k) \asymp \frac{q^{\floor{n/2}}}{k^\delta (1 + \log k)^{3/2}}.
\]
In the $\dagger$-symmetric case, we will prove that
\[
    H^\dagger(n, k) \ll \frac{q^{n/2}}{k^\delta (1 + \log k)^{1/2}}.
\]
See \Cref{prop:mult-problem,rem:mult-table-star-symmetric-case}\phantomref{rem:mult-table-star-symmetric-case} for these results. The latter is only a rough analogue of \eqref{eq:Hnk}, because the exponent of $(1 + \log k)$ is wrong and we do not prove a matching lower bound. Most likely the true analogue holds, but we do not pursue it because we do not need it for the application to derangements. Propositions \ref{prop:ff*}--\ref{prop:bounded_degree} contain all other results that we need.

We note, in particular, that many of the results do not follow trivially just from \eqref{eq:star-symmetric-reduction} and an analogous result for ordinary polynomials. Moreover in the $\dagger$-symmetric case we are not aware of a relation analogous to \eqref{eq:star-symmetric-reduction}.

The starting point of our analysis is a prime polynomial theorem for palindromic polynomials (\Cref{prop:PNT}), which can be proved quite easily. The main technical ingredient is a certain Poisson-type estimate (see \Cref{prop:poisson-tail}), which is directly inspired by an analogue for permutations (\cite{ford--toolkit}*{Theorem 1.5}). Most of the subsequent results rely essentially on this estimate.

\subsection{Notation}

\begin{itemize}[label=$\diamond$]
\item As usual $k(G)$ denotes the number of conjugacy classes of $G$.
\item $\delta(G, \Omega)$ is the proportion of elements of $G$ which are derangements on $\Omega$. $\derang(G, \Omega)$ is the proportion of conjugacy classes of $G$ containing derangements.
\item We may write $\delta(G, H)$ for $\delta(G, \Omega)$ if $G$ is transitive on $\Omega$ and $H$ is a point stabilizer (so we may identify $\Omega$ with $G/H$), and similarly $\derang(G, H)$.
\item Occasionally we use the Iverson bracket $\ind{E}$ to denote the indicator function of an event $E$, such as $\ind{q~\text{odd}}$.
\item For notation related to classical groups and algebraic groups, see \Cref{sec:preliminaries}.
\end{itemize}

\subsection{Acknowledgements} We thank Bob Guralnick for an explanation on \cite{fulman2012_conjugacy_classes}*{Lemma~5.4}.
We are also grateful to Lior Bary-Soroker for drawing our attention to the relationship $f(X) = X^n g(X+1/X)$ between $*$-symmetric polynomials and ordinary polynomials, and to Nick Gill for helpful discussions.

\section{Anatomy of palindromic polynomials}
\label{sec:anatomy}

Given a monic polynomial $f(X) = X^n + a_{n-1} X^{n-1} + \cdots + a_1 X + a_0$ with coefficients in $\F_q$ and $a_0 \ne 0$ we define (as in the introduction)
\[
    f^*(X) = X^n f(1/X) / f(0) = X^n + (a_1/a_0) X^{n-1} + \cdots + (a_{n-1}/a_0) X + (1/a_0).
\]
Note that $(fg)^* = f^* g^*$ provided $f(0), g(0) \ne 0$.
We say $f$ is \emph{$*$-symmetric} if $f^*=f$.
Note this implies $a_0 = \pm 1$.
Similarly, if $q$ is a square we denote by $\bar x = x^{q^{1/2}}$ the involutory automorphism and we define
\[
    f^\dagger(X) = X^n \bar f(1/X) / \bar f(0) = X^n + (\bar a_1 / \bar a_0) X^{n-1} + \cdots + (\bar a_{n-1} / \bar a_0) X + (1 / \bar a_0).
\]
Again $(fg)^\dagger = f^\dagger g^\dagger$ provided $f(0)g(0) \ne 0$, and we say $f$ is \emph{$\dagger$-symmetric} if $f^\dagger = f$.
Note this implies $a_0 \in U$, where $U \le \F_q^\times$ is the cyclic subgroup of order $q^{1/2}+1$.

Let $\P, \P^*, \P^\dagger$ be the sets of monic polynomials over $\F_q$ which are unrestricted, $*$-symmetric, and (if $q$ is a square) $\dagger$-symmetric respectively.
Let $\P(n), \P^*(n), \P^\dagger(n)$ be the sets of those of degree $n$.
We are interested in the anatomies of typical elements of the following sets of polynomials:
\begin{align*}
    & \P_a(n) = \{f \in \P(n) : f(0) = a\} && (a \in \F_q),\\
    & \P_a^*(n) = \{f \in \P^*(n) : f(0) = a\} && (a = \pm 1), \\
    & \P_a^\dagger(n) = \{f \in \P^\dagger(n) : f(0) = a\} && (a \in U).
\end{align*}
For reference we record the sizes of these sets (for $n \ge 1$):
\begin{equation}
\label{eq:P_a-size}
\begin{aligned}
    &|\P_a(n)| = q^{n-1}, \\
    &|\P_a^*(n)| = \begin{cases}
        q^{n/2} &: n~\text{even},~a = +1,\\
        q^{n/2-1} &: n~\text{even},~a = -1,~q~\text{odd}\\
        q^{(n-1)/2} &: n~\text{odd},~a = \pm 1
    \end{cases}\\
    &|\P_a^\dagger(n)| = q^{(n-1)/2}.
\end{aligned}
\end{equation}

In the $*$-symmetric case, it is almost always sufficient to consider $\P_1^*(2n)$, because
\begin{equation}
    \label{eq:star-reductions}
    \begin{aligned}
        &\P_{-1}^*(n) = \{(X-1) f : f \in \P_1^*(n-1)\},\\
        &\P_1^*(2n+1) = \{(X+1) f :f \in \P_1^*(2n)\}.
    \end{aligned}
\end{equation}
Moreover, as mentioned in the introduction, it is easy to see that
\[
    \P_1^*(2n) = \{X^n g(X+1/X) : g \in \P(n)\}.
\]
This relationship will enable us in some cases to reduce questions about $*$-symmetric polynomials to corresponding questions about ordinary polynomials.

\subsection{Prime polynomial theorems}

Let $\pi_q(n,a), \pi^*_q(n,a), \pi^\dagger_q(n,a)$ denote the number irreducible $f \in \P_a(n), \P_a^*(n), \P_a^\dagger(n)$, respectively.
Also let $\pi_q(n) = \sum_{a \in \F_q} \pi_q(n,a)$, etc.
There is a well-known formula $\pi_q(n) = (q^n - O(q^{n/2})) / n$ analogous to the prime number theorem for integers.
There are similar formulas for $\pi^*_q(n)$ and $\pi^\dagger_q(n)$.
The formula for $\pi_q^*(n)$ is due to Carlitz~\cite{carlitz} (see also Cohen~\cite{cohen}).
The formula for $\pi^\dagger_q(n)$ may have first appeared as \cite{fulman1999cycle}*{Theorem~9}.
These formulas can be viewed as concrete cases of Chebotarev's density theorem.

Below we establish similar formulae for $\pi_q(n, a), \pi_q^*(n, a), \pi_q^\dagger(n, a)$.
The first part may be viewed as a special case of the prime polynomial theorem in arithmetic progressions.
Our approach seems to be more direct than those of \cites{carlitz,cohen,fulman1999cycle}.

\begin{proposition}
\label{prop:PNT}
\leavevmode
\begin{enumerate}
\item Let $q$ be a prime power, $n \ge 1$, $a \in \F_q^\times$, and let $r_d(a)$ denote the number of $d$-th roots of $a$ in $\F_q^\times$. Then
\begin{align*}
    \pi_q(n, (-1)^n a)
    = \frac1n \sum_{d \mid n} \mu(d) r_d(a) \frac{q^{n/d}-1}{q-1}
    = \frac{q^n-1 - \eps q^{n/2}}{n(q-1)},
\end{align*}
where $0 \le \eps \le 10$.
\item Let $q$ be a prime power. Apart from $X\pm 1$, all $*$-symmetric monic irreducible polynomials have even degree and constant coefficient $+1$. Their number is given by
\begin{align*}
    \pi^*_q(2n)
    = \frac1{2n} \br{\sum_{\textup{odd}~d \mid n} \mu(d) (q^{n/d} - \eta)}
    = \frac{q^n - \eps q^{n/3}}{2n},
\end{align*}
where $0 \le \eps \le 3$ and $\eta = \ind{q~\textup{odd}}$.
\item Let $q$ be a square prime power. All $\dagger$-symmetric monic irreducible polynomials have odd degree and constant coefficient in $U$. Conversely for any $a \in U$ we have
\begin{align*}
    \pi^\dagger_q(2n-1, -a)
    = \frac1{2n-1} \sum_{d \mid 2n-1} \mu(d) r_d(a) \frac{q^{(n-1/2) / d} + 1}{q^{1/2} + 1}
    = \frac{q^{n-1/2} + 1 - \eps q^{n/3-1/6}}{(2n-1)(q^{1/2}+1)},
\end{align*}
where $r_d(a)$ is the number of $d$-th roots of $a$ in $U$.
Here $0 \le \eps \le 40$.
\end{enumerate}
\end{proposition}
\begin{proof}
\begin{enumerate}[wide]
\item
Let $\norm : \F_{q^n} \to \F_q$ be the norm map.
If $e \mid n$, the number of elements $x \in \F_{q^e}$ such that $\norm(x) = a$ is $r_{n/e}(a) (q^e-1)/(q-1)$, where $r_d(a)$ is the number of $d$-th roots of $a$ in $\F_q$.
Hence if $F(e)$ is the number of such $x$ of degree $e$ (i.e., not contained in a proper subfield) then
\begin{equation}
    \label{eq:sum-F(d)}
    \sum_{d \mid e} F(d) = r_{n/e}(a) (q^e - 1)/(q-1).
\end{equation}
Hence by M\"obius inversion
\[
    F(n) = \sum_{d \mid e} \mu(d) r_d(a) (q^{n/d}-1) / (q-1).
\]
Dividing by $n$ gives the number of monic irreducible polynomials of degree $n$ with constant coefficient $(-1)^n a$.
This proves the claimed formula.
For the estimate, note that $F(n) \le (q^n-1)/(q-1)$ (by \eqref{eq:sum-F(d)}) and
\[
    \left|\sum_{d \mid n, d < n} \mu(d) r_d(a) (q^{n/d}-1)\right| \le \sum_{d \mid n, d < n} d q^{n/d}
    \le 2 q^{n/2} + n q^{n/3} / (1-1/q)
    \le 10 q^{n/2}.
\]

\item
We follow a similar pattern.
The minimal polynomial of $x \in \bar \F_q^\times$ is $*$-symmetric if and only if $x$ is conjugate to $x^{-1}$ under the Galois group, i.e., if and only if
\[
    x^{q^n} = x^{-1}
\]
for some $n \ge 0$.
If $n = 0$ then $x = \pm 1$.
Otherwise, the norm of $x$ is $1$, being the product of a collection of pairs $y,y^{-1}$,
and if $n$ is minimal then the degree of $x$ is $2n$, for clearly
\[
    x^{q^{2n}} = x^{-q^n} = x,
\]
and if $x^{q^m} = x$ then
\[
    x^{q^{m-n}} = (x^{-q^n})^{q^{m-n}} = x^{-q^m} = x^{-1},
\]
which implies $n \le m-n$, i.e., $m \ge 2n$.
Moreover, if $n$ is minimal then $x^{q^m} = x^{-1}$ if and only if $m$ is an odd multiple of $n$.

Let $F^*(n)$ be the number of such $x$ of degree $2n$.
If $q$ is odd there are exactly $q^n-1$ solutions to $x^{q^n+1} = 1$ apart from $\pm1$,
while if $q$ is even there are $q^n$ solutions apart from $1$, so
\[
    \sum_{\text{odd}~d \mid n} F^*(n/d) = q^n - \eta.
\]
Applying M\"obius inversion,
\[
    F^*(n) = \sum_{\text{odd}~d \mid n} \mu(d) (q^{n/d} - \eta).
\]
Dividing by $2n$ gives the number of $*$-symmetric monic irreducible polynomials of degree $2n$.
This proves the exact formula, and the estimate follows as in case (1).

\item
The minimal polynomial of $x \in \bar \F_q^\times$ is $\dagger$-symmetric if and only if
\begin{equation}
    \label{eq:1/xbar}
    x^{q^n} = x^{-q^{1/2}}
\end{equation}
for some $n \ge 1$.
Note this is equivalent to $x^{q^{n-1/2}} = x^{-1}$.
If $n$ is minimal then the degree of $x$ is $2n-1$, for
\[
    x^{q^{2n-1}} = x^{-q^{n-1/2}} = x,
\]
and if $x^{q^m} = x$ then
\[
    x^{q^{m-n+1}} = (x^{-q^{n-1/2}})^{q^{m-n+1}} = x^{-q^{1/2}},
\]
which implies $n \le m - n + 1$, i.e., $m \ge 2n-1$.
Moreover, if $n$ is minimal then $x^{q^m} = x^{-q^{1/2}}$ if and only if $2n-1 \mid 2m-1$.

Let $\norm : \F_{q^{2n-1}} \to \F_q$ be the norm map.
Explicitly, $\norm(x) = x^{1 + q + \cdots + q^{2n-2}} = x^{(q^{2n-1} - 1)/(q-1)}$.
The subgroup $U_n$ of $\F_{q^{2n-1}}^\times$ consisting of solutions to \eqref{eq:1/xbar} is cyclic of order $q^{n-1/2}+1$.
Now note that
\[
    \frac{q^{2n-1}-1}{q-1} = \frac{q^{n-1/2}-1}{q^{1/2}-1} \cdot \frac{q^{n-1/2}+1}{q^{1/2}+1},
\]
and the first factor is prime to $q^{n-1/2}+1$, so $\norm$ maps $U_n$ onto the cyclic group $U = U_1$ of order $q^{1/2}+1$.
Thus for every $a \in U$ there are exactly $(q^{n-1/2}+1)/(q^{1/2}+1)$ elements $x \in U_n$ such that $\norm(x) = a$.

Fix any such $a$. For $e \mid 2n-1$, the number of elements $x \in \F_{q^e} \cap U_n$ such that $\norm(x) = a$ is $r_{(2n-1)/e}(a) (q^{e/2}+1)/(q^{1/2}+1)$, where $r_d(a)$ is the number of $d$-th roots of $a$ in $V$.
Hence if $F^\dagger(e)$ is the number of such $x$ of degree $e$ then
\[
    \sum_{d \mid e} F^\dagger(d) = r_{(2n-1)/e}(a) (q^{e/2} + 1)/(q^{1/2} + 1).
\]
Applying M\"obius inversion,
\[
    F^\dagger(2n-1) = \sum_{d \mid 2n-1} \mu(d) r_d(a) (q^{(n-1/2)/d} + 1) / (q^{1/2} + 1).
\]
Like in the previous two cases, we get the number of $\dagger$-symmetric monic irreducible polynomials of degree $2n-1$, with constant coefficient $-a$, by dividing by $2n-1$, and we can estimate the sum straightforwardly by comparing with a geometric series.\qedhere
\end{enumerate}
\end{proof}

\subsection{Poisson-type estimates for the number of irreducible factors}

\def\nstar{{\not{\, *}}}
\def\ndagger{{\not{\, \dagger}}}

Next we prove Poisson-type tail estimates for the number of square-free polynomials with exactly a given number of irreducible factors in specified sets.

Let $H_n = 1 + 1/2 + \cdots + 1/n = \log n + O(1)$ denote the harmonic sum.
Also define $H_n^{\text{even}} = \sum_{\text{even}~k \le n} 1/k = H_{\floor{n/2}} / 2$ and $H_n^\text{odd} = \sum_{\text{odd}~k \le n} 1/k = H_n - H_{\floor{n/2}}/2$.

Let $\I_{\le n}$ be the set of irreducible polynomials of degree $\le n$.
For $I \subset \I_{\le n}$ let
\[
    H(I) = \sum_{f \in I} q^{-\deg f}.
\]
We need one calculation in particular:
\[
    H(\I_{\le n}) = \sum_{d=1}^n \frac{\pi_q(d)}{q^d}
    = \sum_{d=1}^n \frac{1}{d} (1 + O(q^{-d/2}))
    = \log n + O(1).
\]

A $*$-symmetric polynomial $f$ is called \emph{$*$-irreducible} if it has no proper $*$-symmetric divisor of positive degree; in other words, if it is irreducible or of the form $gg^*$ with $g$ irreducible and non-$*$-symmetric.
We use the term $\dagger$-irreducible similarly.

Let $\I^*_{\le n}$ be the set of $*$-irreducible polynomials of degree $\le n$.
For $I \subset \I_{\le n}^*$ let
\[
    H^*(I) = \sum_{f \in I} q^{-(\deg f)/2}.
\]
For reference, from \Cref{prop:PNT}(2),
\begin{align*}
    H^*(\I^*_{\le n})
    &= \sum_{d=1}^{n/2} \frac{\pi_q^*(2d)}{q^d} + \sum_{d=1}^{n/2} \frac{\pi_q^\nstar(d)/2}{q^d} + (1 + \eta) q^{-1/2}\\
    &= \sum_{d=1}^{n/2} \frac{1}{2d} (1 + O(q^{-2d/3})) + \sum_{d=1}^{n/2} \frac{1}{2d} (1 + O(q^{-d/2}))
    = \log n + O(1).
\end{align*}
Here $\pi_q^\nstar(d) = \pi_q(d) - \pi_q^*(d) - \ind{d=1}$ denotes the number of non-$*$-symmetric irreducible polynomials of degree $d$ apart from $X$.
Define $\I^\dagger_{\le n}$ and $H^\dagger(I)$ for $I \subset \I^\dagger_{\le n}$ similarly (if $q$ is a square), and again
\begin{align*}
    H^\dagger(\I^\dagger_{\le n})
    &= \sum_{d=1}^n \frac{\pi_q^\dagger(d)}{q^{d/2}} + \sum_{d=1}^{n/2} \frac{\pi_q^\ndagger(d)/2}{q^d}\\
    &= \sum_{\text{odd}~d \le n} \frac{1}{d} (1 + O(q^{-d/3})) + \sum_{d=1}^{n/2} \frac{1}{2d} (1 + O(q^{-d/2}))
    = \log n + O(1).
\end{align*}
Here $\pi_q^\ndagger(d) = \pi_q(d) - \pi_q^\dagger(d) - \ind{d=1}$.

The following proposition is modelled after \cite{ford--toolkit}*{Theorem~1.5}.

\begin{proposition}
\label{prop:poisson-tail}
We have the following Poisson-type tail bounds.
\begin{enumerate}
\item
Let $I_1, \dots, I_r$ be an arbitrary partition of $\I_{\le n}$ and let $m_1, \dots, m_r \ge 0$.
The number of square-free $f \in \P_a(n)$ (for any $a \in \F_q^\times$) with exactly $m_i$ factors in $I_i$ for each $i$ is
\[
    \ll q^{n-1} \prod_{i=1}^r \br{e^{-H(I_i)} \frac{H(I_i)^{m_i}}{m_i!}} \br{\frac{m_1}{H(I_1)} + \cdots + \frac{m_r}{H(I_r)}}.
\]

\item
Let $I_1, \dots, I_r$ be an arbitrary partition of $\I^*_{\le 2n}$ and let $m_1, \dots, m_r \ge 0$.
The number of square-free $f \in \P^*_1(2n)$ with exactly $m_i$ factors in $I_i$ for each $i$ is
\[
    \ll q^n \prod_{i=1}^r \br{e^{-H^*(I_i)} \frac{H^*(I_i)^{m_i}}{m_i!}} \br{\frac{m_1}{H^*(I_1)} + \cdots + \frac{m_r}{H^*(I_r)}}.
\]

\item
Let $I_1, \dots, I_r$ be an arbitrary partition of $\I^\dagger_{\le n}$ and let $m_1, \dots, m_r \ge 0$.
The number of square-free $f \in \P^\dagger_a(n)$ (for any $a \in U$) with exactly $m_i$ factors in $I_i$ for each $i$ is
\[
    \ll q^{(n-1)/2} \prod_{i=1}^r \br{e^{-H^\dagger(I_i)} \frac{H^\dagger(I_i)^{m_i}}{m_i!}} \br{\frac{m_1}{H^\dagger(I_1)} + \cdots + \frac{m_r}{H^\dagger(I_r)}}.
\]
\end{enumerate}
\end{proposition}
\begin{proof}
\begin{enumerate}[wide]
\item
Let $S_a(n; m_1, \dots, m_r)$ be the set of square-free $f \in \P_a(n)$ with exactly $m_i$ factors in $I_i$ for each $i$.
Let $S(n; m_1, \dots, m_r) = \bigcup_{a \in \F_q^\times} S_a(n ; m_1, \dots, m_r)$.
We can specify an element of $S(n; m_1, \dots, m_r)$ uniquely by first choosing nonnegative integers $m_{i,d}$ ($1 \le i \le r$, $1 \le d \le n$) subject to
\begin{equation}
\label{eq:poisson-tail-1}
\begin{aligned}
&&\sum_{d=1}^n m_{i,d} = m_i ~ (1 \le i \le r),
&&\sum_{i=1}^r \sum_{d=1}^n d m_{i,d} = n,
\end{aligned}
\end{equation}
and then for each $i, d$ choosing $m_{i,d}$ distinct factors in $I_i^{(d)}$, where $I_i^{(d)}$ is the set of elements of $I_i$ of degree $d$ (excluding $X$ if $d=1$).
Thus
\begin{equation}
    \label{eq:S-exact}
    S(n; m_1, \dots, m_r) = \sum_{\eqref{eq:poisson-tail-1}} \prod_{i=1}^r \prod_{d=1}^n \binom{|I_i^{(d)}|}{m_{i,d}}.
\end{equation}
To use this formula effectively we first apply a sum-smoothing trick.
By isolating an irreducible factor $g \mid f$ we have
\begin{align*}
    |S_a(n; m_1, \dots, m_r)| n
    &= \sum_{f \in S_a(n; m_1, \dots, m_r)} \sum_{\text{irred.}~g \mid f} \deg g \\
    &= \sum_{g \in \I_{\le n}} |\{f \in S_a(n; m_1, \dots, m_r) : g \mid f\}| \deg g.
\end{align*}
Now the $g$-divisible elements of $S_a(n; m_1, \dots, m_r)$ are obviously in one-to-one correspondence with the $g$-indivisible elements of $S_{a/g(0)}(n - \deg g; (m_i - \ind{g \in I_i})_{i=1}^r)$.
Thus by ignoring the $g$-indivisible restriction we get
\[
    |S_a(n; m_1, \dots, m_r)| n
    \le \sum_{g \in \I_{\le n} \sm \{X\}} |S_{a / g(0)}(n - \deg g; (m_i - \ind{g \in I_i})_{i=1}^r)| \deg g.
\]
By \Cref{prop:PNT}(1), $\pi_q(d, b) \le q^d / d (q-1)$, so
\begin{align*}
    |S_a(n; m_1, \dots, m_r) n
    &\le \frac{q}{q-1} \sum_{d=1}^n \sum_{b\in \F_q^\times} \sum_{j=1}^r q^{d-1} |S_{a/b}(n - d; (m_i - \ind{i=j})_{i=1}^r)| \\
    &= \frac{q}{q-1} \sum_{d=1}^n \sum_{j=1}^r q^{d-1} |S(n - d; (m_i - \ind{i=j})_{i=1}^r)|.
\end{align*}
Now applying \eqref{eq:S-exact}, we get
\begin{align*}
    |S_a(n; m_1, \dots, m_r)| n
    &\le \frac{q}{q-1} \sum_{d=1}^n \sum_{j=1}^r q^{d-1} \sum_{\substack{
        \sum_{e=1}^n m_{i, e} = m_i - \ind{i=j}~(1 \le i \le r), \\
        \sum_{i=1}^r \sum_{e=1}^n e m_{i,e} = n - d
    }} \prod_{i=1}^r \prod_{e=1}^n \binom{|I_i^{(e)}|}{m_{i,e}} \\
    &\le \frac{q}{q-1} \sum_{d=1}^n \sum_{j=1}^r q^{d-1} \sum_{\substack{
        \sum_{e=1}^n m_{i, e} = m_i - \ind{i=j}~(1 \le i \le r), \\
        \sum_{i=1}^r \sum_{e=1}^n e m_{i,e} = n - d
    }} \prod_{i=1}^r \prod_{e=1}^n \frac{|I_i^{(e)}|^{m_{i,e}}}{m_{i,e}!} \\
    &\le \frac{q^n}{q-1} \sum_{d=1}^n \sum_{j=1}^r \sum_{\substack{
        \sum_{e=1}^n m_{i, e} = m_i - \ind{i=j}~(1 \le i \le r), \\
        \sum_{i=1}^r \sum_{e=1}^n e m_{i,e} = n - d
    }} \prod_{i=1}^r \prod_{e=1}^n \frac{(|I_i^{(e)}| / q^e)^{m_{i,e}}}{m_{i,e}!} \\
    &\le \frac{q^n}{q-1} \sum_{j=1}^r \sum_{\substack{
        \sum_{e=1}^n m_{i, e} = m_i - \ind{i=j}~(1 \le i \le r)
    }} \prod_{i=1}^r \prod_{e=1}^n \frac{(|I_i^{(e)}| / q^e)^{m_{i,e}}}{m_{i,e}!}.
\end{align*}
Finally, applying the multinomial theorem, we get
\begin{align*}
    |S_a(n ; m_1, \dots, m_r)|n
    &\le \frac{q^n}{q-1} \sum_{j=1}^r \prod_{i=1}^r \frac{\br{\sum_{e=1}^n |I_i^{(e)}| / q^e}^{m_i - \ind{i=j}}}{(m_i - \ind{i=j})!} \\
    &= \frac{q^n}{q-1} e^{H(\I_{\le n})} \prod_{i=1}^r \br{e^{-H(I_i)} \frac{H(I_i)^{m_i}}{m_i!}} \br{\frac{m_1}{H(I_1)} + \cdots + \frac{m_r}{H(I_r)}}.
\end{align*}
Finally, we use the fact that $H(\I_{\le n}) = \log n + O(1)$.

\item We could give a completely analogous argument. Alternatively, we can just use \eqref{eq:star-symmetric-reduction}.
Note that a $*$-symmetric polynomial $f(X) = X^n g(X+1/X)$ (of degree $2n$ and constant coefficient $1$) is square-free if and only if $g$ is square-free and has no factors of $X \pm 2$, and in this case $f$ has no factors of $X \pm 1$. Moreover $f$ is $*$-irreducible if and only if $g$ is irreducible.
Therefore we get the result directly from \eqref{eq:star-symmetric-reduction} and part (1).

\item We give an argument similar to that in part (1).
Let $S_a(n; m_1, \dots, m_r)$ be the set of square-free $f \in \P_a^\dagger(n)$ with exactly $m_i$ factors in $I_i$ for each $i$. Then
\begin{align*}
    |S_a(n ; m_1, \dots, m_r)| n
    &= \sum_{f \in S_a(n; m_1, \dots, m_r)} \sum_{\dagger\text{-irred}~g \mid f} \deg g\\
    &\le \sum_{g \in \I^\dagger_{\le n}} |S_{a / g(0)}(n - \deg g; (m_i - \ind{g \in I_i})_{i=1}^r)| \deg g\\
    &\le \sum_{d=1}^n \sum_{b \in U} \sum_{j=1}^r \pi_q^\dagger(d, b) |S_{a/b}(n - d ; (m_i - \ind{i=j})_{i=1}^r)| d\\
    &\le \sum_{d=1}^n \sum_{b \in U} \sum_{j=1}^r q^{d/2-1/2} |S_{a/b}(n - d ; (m_i - \ind{i=j})_{i=1}^r)| \\
    &\le \sum_{d=1}^n \sum_{j=1}^r q^{d/2-1/2} \sum_{\substack{
        \sum_{e=1}^n m_{i,e} = m_i - \ind{i=j} ~ (1 \le i \le r) \\
        \sum_{i=1}^r \sum_{e=1}^n e m_{i,e} = n - d
    }} \prod_{i=1}^r \prod_{e=1}^n \binom{|I_i^{(e)}|}{m_{i,e}} \\
    &\le \sum_{d=1}^n \sum_{j=1}^r q^{d/2-1/2} \sum_{\substack{
        \sum_{e=1}^n m_{i,e} = m_i - \ind{i=j} ~ (1 \le i \le r) \\
        \sum_{i=1}^r \sum_{e=1}^n e m_{i,e} = n - d
    }} \prod_{i=1}^r \prod_{e=1}^n \frac{|I_i^{(e)}|^{m_{i,e}}}{m_{i,e}!} \\
    &\le q^{n/2-1/2} \sum_{d=1}^{n/2} \sum_{j=1}^r \sum_{\substack{
        \sum_{e=1}^n m_{i,e} = m_i - \ind{i=j} ~ (1 \le i \le r) \\
        \sum_{i=1}^r \sum_{e=1}^n e m_{i,e} = n - d
    }} \prod_{i=1}^r \prod_{e=1}^n \frac{(|I_i^{(e)}| / q^{e/2})^{m_{i,e}}}{m_{i,e}!} \\
    &\le q^{n/2-1/2} \sum_{j=1}^r \sum_{\substack{
        \sum_{e=1}^n m_{i,e} = m_i - \ind{i=j} ~ (1 \le i \le r)
    }} \prod_{i=1}^r \prod_{e=1}^n \frac{(|I_i^{(e)}| / q^{e/2})^{m_{i,e}}}{m_{i,e}!} \\
    &\le q^{n/2-1/2} e^{H^\dagger(\I^\dagger_{\le n})} \prod_{i=1}^r \br{e^{-H^\dagger(I_i)} \frac{H^\dagger(I_i)^{m_i}}{m_i!}} \br{\frac{m_1}{H^\dagger(I_1)} + \cdots + \frac{m_r}{H^\dagger(I_r)}}.
\end{align*}
Now as before we use $H^\dagger(\I^\dagger_{\le n}) = \log n + O(1)$.
\qedhere
\end{enumerate}
\end{proof}

\begin{corollary}
\label{cor:poisson-tail}
\begin{enumerate}
\item
Let $a \in \F_q^\times$, $\ell \ge 0$, and $1 \le k \le n$.
The number of square-free $f \in \P_a(n)$ with exactly $\ell$ irreducible factors of degree $\le k$ is
\[
    \ll \frac{q^{n-1}}{k}  \frac{H_k^\ell}{\ell!}
    \br{1 + \frac{\ell}{H_k}} .
\]
\item
Let $\ell_1, \ell_2 \ge 0$ and $1 \le k_1, k_2 \le 2n$.
The number of square-free $f \in \P^*_1(2n)$ with exactly $\ell_1$ $*$-symmetric irreducible factors of degree $\le k_1$ and $\ell_2$ pairs of non-$*$-symmetric irreducible factors of degree $\le k_2$ is
\[
    \ll
    \frac{q^n}{k_1^{1/2} k_2^{1/2}} \frac{(H_{k_1}^{\textup{even}})^{\ell_1}}{\ell_1!} \frac{(H_{k_2} / 2)^{\ell_2}}{\ell_2!}
    \br{
        1 + \frac{\ell_1}{H_{k_1}} + \frac{\ell_2}{H_{k_2}}
    }.
\]
\item
Let $a \in U$, $\ell_1, \ell_2 \ge 0$, and $1 \le k_1, k_2 \le n$.
The number of square-free $f \in \P^\dagger_a(n)$ with exactly $\ell_1$ $\dagger$-symmetric irreducible factors of degree $\le k_1$ and $\ell_2$ pairs of non-$\dagger$-symmetric irreducible factors of degree $\le k_2$ is
\[
    \ll \frac{q^{n/2-1/2}}{k_1^{1/2} k_2^{1/2}} \frac{(H_{k_1}^\textup{odd})^{\ell_1}}{\ell_1!} \frac{(H_{k_2} / 2)^{\ell_2}}{\ell_2!}
    \br{1 + \frac{\ell_1}{H_{k_1}} + \frac{\ell_2}{H_{k_2}}}.
\]
\end{enumerate}
\end{corollary}
\begin{proof}
Each of these follows directly from \Cref{prop:poisson-tail} by taking the appropriate partition and summing the uninteresting variable.
For example, let us prove (2).
Let $I_1 \subset \I_{\le 2n}^*$ be the set of $*$-symmetric irreducible polynomials of degree $\le k_1$ (excluding $X\pm 1$),
let $I_2 \subset \I_{\le 2n}^*$ be the set of $*$-irreducible products $gg^*$ with $\deg g \le k_2$,
and let $I_3 = \I_{\le 2n}^* \sm (I_1 \cup I_2)$.
By \Cref{prop:poisson-tail}, the number of square-free $f \in \P_1^*(2n)$ with $\ell_1$ factors in $I_1$, $\ell_2$ factors in $I_2$, and $\ell_3$ factors in $I_3$ is
\[
    \ll q^n \prod_{i=1}^3 \br{ e^{-H^*(I_i)} \frac{H^*(I_i)^{\ell_i}}{\ell_i!}} \br{\frac{\ell_1}{H^*(I_1)} + \frac{\ell_2}{H^*(I_2)} + \frac{\ell_3}{H^*(I_3)}}.
\]
The sum over all $\ell_3 \ge 0$ is
\[
    \ll q^n \prod_{i=1}^2 \br{ e^{-H^*(I_i)} \frac{H^*(I_i)^{\ell_i}}{\ell_i!}} \br{\frac{\ell_1}{H^*(I_1)} + \frac{\ell_2}{H^*(I_2)} + 1}.
\]
Now by \Cref{prop:PNT}(2) we have
\begin{align*}
    H^*(I_1) &= \sum_{2d \le k_1} \frac{\pi^*_q(2d)}{q^d} = H^\text{even}_{k_1} - \eps_1, \\
    H^*(I_2) &=  \sum_{d \le k_2} \frac{(\pi_q(d) - \pi^*_q(d))/2}{q^d} = H_{k_2} / 2 - \eps_2,
\end{align*}
where the errors satisfy $0 \le \eps_1, \eps_2 \ll 1$.
Therefore we can replace $H^*(I_1)$ by $H_{k_1}^\text{even}$ and $H^*(I_2)$ by $H_{k_2}/2$, and we get the claimed bound.
\end{proof}

\subsection{Polynomials with a factor of a given degree}

Write $H_a(n,k), H_a^*(n,k), H_a^\dagger(n,k)$ for the number of $f \in \P_a(n), \P_a^*(n), \P_a^\dagger(n)$ (respectively) having an unrestricted, $*$-symmetric, $\dagger$-symmetric (respectively) factor of degree $k$.
%We use $\tilde H_a(n,k), \tilde H_a^*(n,k), \tilde H_a^\dagger(n,k)$ to denote the number of such $f$ which are also square-free.

\begin{proposition}
\label{prop:mult-problem}
Let $\delta = 1 - (1 + \log \log 2) / \log 2 \approx 0.086$.
Let $1 \le k \le n/2$.
\begin{enumerate}
\item For $a \in \F_q^\times$,
\[
    H_a(n, k) \ll q^{n-1} k^{-\delta} (1+\log k)^{-1/2}.
\]
\item For $a = \pm 1$,
\[
    H^*_a(n, k) \ll |\P_a^*(n)| k^{-\delta} (1 + \log k)^{-1/2}.
\]
\item For $a \in U$,
\[
    H^\dagger_a(n, k) \ll q^{(n-1)/2} k^{-\delta} (1 + \log k)^{-1/2}.
\]
\end{enumerate}
\end{proposition}
\begin{proof}
The three arguments are similar. For simplicity we just give the third.
Since the claim is trivial for bounded $k$ we may assume $k \ge 10$, say.
Let $\tilde H_a^\dagger(n, k)$ denote the number of \emph{square-free} $f \in \P_a^\dagger(n)$ having a $\dagger$-symmetric factor of degree $k$.
By \Cref{cor:poisson-tail}(3) and the binomial theorem, the number of square-free $f \in \P_a^\dagger(n)$ having exactly $\ell$ $\dagger$-irreducible factors of degree $\le k$
\[
    \ll \frac{q^{n/2-1/2}}{k} \frac{H_k^\ell}{\ell!} \br{1 + \frac{\ell}{H_k}}.
\]
The sum of this over all $\ell \ge \log_2 k$ is
\[
    \ll \frac{q^{n/2-1/2}}{k} \sum_{\ell \ge \log_2 k} \frac{H_k^{\ell-1}}{(\ell-1)!}.
\]
Let $\ell_0 = \ceil{\log_2 k}$. Then $H_k / \ell_0 < 0.9$ (since $k \ge 10$), so by comparing with a geometric series and using Stirling's approximation we have
\[
    \sum_{\ell \ge \ell_0} \frac{H_k^{\ell-1}}{(\ell-1)!} \ll \frac{H_k^{\ell_0-1}}{(\ell_0-1)!} \asymp \frac{H_k^{\ell_0}}{\ell_0!}.
\]
Applying Stirling's approximation and using $H_k / \ell_0 = \log 2 + O(1/\log k)$, this is
\[
    \asymp (e H_k / \ell_0)^{\ell_0} \ell_0^{-1/2}
    \asymp (e \log 2)^{\log_2 k} (\log k)^{-1/2}  = k^{(1 + \log\log 2) / \log 2} (\log k)^{-1/2}.
\]
Hence the number of square-free $f \in \P_a^\dagger(n)$ having exactly $\ell$ $\dagger$-irreducible factors of degree $\le k$ is $\ll q^{n/2-1/2} k^{-\delta} (\log k)^{-1/2}$.

On the other hand the number of square-free $f \in \P_a^\dagger(n)$ factorizing as $f_1f_2$ where $f_i \in \P_{a_i}^\dagger(n_i)$ has $\ell_i$ $\dagger$-irreducible factors of degree $\le k$ for $i=1,2$, where $(n_1, n_2) = (k, n-k)$, is
\[
    \ll \frac{q^{k/2-1/2}}{k} \frac{H_k^{\ell_1}}{\ell_1!} \br{1 + \frac{\ell_1}{H_k}}
    \frac{q^{(n-k)/2-1/2}}{k} \frac{H_k^{\ell_2}}{\ell_2!} \br{1 + \frac{\ell_2}{H_k}},
\]
and the sum of this over all $\ell_1, \ell_2 \ge 0$ and $a_1, a_2 \in U$ such that $\ell_1+\ell_2=\ell$ and $a_1a_2 = a$ is
\[
    \ll \frac{q^{n/2-1/2}}{k^2} \frac{(2 H_k)^\ell}{\ell!} \br{1 + \frac{\ell}{H_k}}^2,
\]
and the sum of this over all $\ell < \log_2 k$ is $\ll q^{n/2-1/2} k^{-\delta} (\log k)^{-1/2}$,
by a similar sequence of approximations as above.
Thus we get the square-free bound $\tilde H_a^\dagger(n, k) \ll q^{(n-1)/2} k^{-\delta} (\log k)^{-1/2}$.

Now an arbitrary polynomial $f \in \P_a^\dagger(n)$ can be written uniquely $f = f_1f_2^2$ with $f_1$ square-free.
Since $f_1$ is just the product of the $\dagger$-symmetric irreducible factors of odd multiplicity, $f_1$ and $f_2$ are $\dagger$-symmetric.
Also we must have $f_1(0)f_2(0)^2 = a$, so $f_1 \in \P_{a / f_2(0)^2}^\dagger(n - 2 \deg f_2)$.
The number of such $f$ with $\deg f_2 \ge \log k$ is bounded by
\[
    \sum_{d \ge \log k} \sum_{b \in U} \sum_{f_2 \in \P_b^\dagger(d)} q^{(n-2d-1)/2}
    \ll \sum_{d \ge \log k} q^{(n-1)/2 - d/2}
    \ll q^{(n-1)/2 - (\log k)/2}.
\]
Hence assume $\deg f_2 < \log k$.
If $f$ has a $\dagger$-symmetric divisor of degree $k$ then there is some $\dagger$-symmetric $g \mid f_2^2$ such that $f_1$ has a $\dagger$-symmetric divisor of degree $k - \deg g = k - O(\log k)$.
It follows that
\begin{align*}
    H_a^\dagger(n,k)
    &\le \sum_{\substack{d < \log k\\b \in U\\ f_2 \in \P_b^\dagger(d)}} \sum_{\substack{g \mid f_2^2 \\ g^\dagger = g}}
    \tilde H^\dagger_{a / b^2}(n - 2d, k - \deg g) + O(q^{(n-1)/2 - (\log k)/2}) \\
    &\le \sum_{\substack{d < \log k\\b \in U\\ f_2 \in \P_b^\dagger(d)}} \sum_{\substack{g \mid f_2^2 \\ g^\dagger = g}}
    q^{(n-2d-1)/2} (k - \deg g)^{-\delta} (\log (k - \deg g))^{-1/2} + O(q^{(n-1)/2 - (\log k)/2}) \\
    &\ll q^{(n-1)/2} k^{-\delta} (\log k)^{-1/2} \sum_{f_2^\dagger=f_2} \frac{d^\dagger(f_2^2)}{q^{\deg f_2}} + O(q^{(n-1)/2 - (\log k)/2}).
\end{align*}
Here $d^\dagger(f)$ denotes the number of $\dagger$-symmetric divisors of $f$.

To complete the proof it suffices to prove that $\sum_{f^\dagger=f} d^\dagger(f^2) / q^{\deg f} \ll 1$, and this follows from an easy Euler product argument:
\begin{align*}
    \sum_{f^\dagger = f} \frac{d^\dagger(f^2)}{q^{\deg f}}
    &= \prod_{g~\dagger\text{-irreducible}} \br{1 + \frac{3}{q^{\deg g}} + \frac{5}{q^{2\deg g}} + \cdots} \\
    &\le \exp \sum_{g~\dagger\text{-irreducible}} \br{\frac{3}{q^{\deg g}} + \frac{5}{q^{2\deg g}} + \cdots} \\
    &\le \exp \sum_{d \ge 1} \sum_{m \ge 1} \frac{2m+1}{q^{md - d/2}} \ll 1.\qedhere
\end{align*}
\end{proof}

\begin{remark}
    \label{rem:mult-table-star-symmetric-case}
    By analogy with the corresponding results for integers, permutations, and ordinary polynomials, we expect that each instance of $O(k^{-\delta} (1 + \log k)^{-1/2})$ in \Cref{prop:mult-problem} can be improved to $\Theta(k^{-\delta} (1 + \log k)^{-3/2})$.
    This is certainly true for part (2), the $*$-symmetric case.
    By \eqref{eq:star-reductions} it suffices to consider the case of $H^*_1(2n)$.
    Every polynomial $f \in \P_1^*(2n)$ can be written uniquely in the form $f(X) = X^n g(X + 1/X)$ where $g \in \P(n)$.
    If $g$ has a factor of degree $k$ then $f$ has a $*$-symmetric factor of degree $2k$, so
    \[
        H_1^*(2n, 2k) \ge H(n, k) \asymp \frac{q^n}{k^\delta (1 + \log k)^{3/2}}
    \]
    by \eqref{eq:Hnk}.
    On the other hand if $f$ is moreover square-free then $f$ cannot have any factors of $X\pm 1$,
    and if $f$ has a $*$-symmetric factor of (necessarily even) degree $2k$ then $g$ has a factor of degree $k$.
    Thus the number of square-free $f \in \P_1^*(2n)$ having a $*$-symmetric factor of degree $2k$ is
    \[
        \tilde H_1^*(2n, 2k) \le H(n, k) \asymp \frac{q^n}{k^\delta (1 + \log k)^{3/2}}
    \]
    by \eqref{eq:Hnk} again.
    The general case reduces to the square-free case as in the proof of \Cref{prop:mult-problem}.
\end{remark}

We will also need the following proposition about polynomials nearly factorizing as $gg^*$ or $gg^\dagger$.

\begin{proposition}
\label{prop:ff*}
\leavevmode
\begin{enumerate}
\item (This item intentionally left blank in order to keep the numbering consistent with the other results in this section.)
\item The number of polynomials $f \in \P_a^*(n)$ ($a=\pm 1$) which factor as $gg^*h$ with $\deg h \le m \le n^{1/2}$ is $\ll |\ca P^*_a(n)|(m+1) / n^{1/2}$.
\item The number of polynomials $f \in \P_a^\dagger(n)$ ($a\in U$) which factor as $gg^\dagger h$ with $\deg h \le m \le n^{1/2}$ is $\ll q^{(n-1)/2} (m+1) / n^{1/2}$.
\end{enumerate}
\end{proposition}
\begin{proof}
\begin{enumerate}[wide,start=2]
\item
By \eqref{eq:star-reductions}, it is enough to prove the statement in the case of $\P_1^*(2n)$. First consider the $h=1$ case.
By \Cref{cor:poisson-tail}(2) with $k_1 = k_2 = 2n$, $\ell_1 = 0$, and summing over $\ell_2$, the number of square-free polynomials $f \in \P_1^*(2n)$ with no $*$-symmetric irreducible factors is $\ll q^n / n^{1/2}$.
Now an arbitrary $f \in \P_1^*(2n)$ can be written uniquely $f = f_1 f_2^2$ where $f_1$ is square-free, and if $f$ is $*$-symmetric then so are $f_1$ and $f_2$.
Moreover, $f$ can be written as $gg^*$ with if and only if $f_1$ has no $*$-symmetric irreducible factors.
Hence, by considering all possibilities for $f_2$, the number of such $f$ is
\[
    \ll \sum_{d < n} q^{d/2} q^{n-d} / (n-d)^{1/2} + q^{n/2} \ll q^n / n^{1/2}.
\]
Now by considering all possibilities for $h \in \P_1^*(2d)$ ($0 \le 2d \le m$) it follows that the number of $f \in \P_1^*(2n)$ factoring as $gg^*h$ is
\[
    \ll \sum_{d=0}^{m} q^{d} q^{n-d} / (n-2d)^{1/2} \asymp (m+1) q^n / n^{1/2}.
\]
\item Similar.\qedhere
%By \Cref{cor:poisson-tail}(3) with $k_1 = k_2 = n$, $\ell_1 = 0$, and summing over $\ell_2$, for any $a \in U$ the number of square-free polynomials $f \in \P_a^\dagger(n)$ with no $\dagger$-symmetric irreducible factors is $\ll q^{n/2-1/2} / n^{1/2}$.
%Now an arbitrary $f$ can be written uniquely $f = f_1f_2^2$ where $f_1$ is square-free, and if $f$ is $\dagger$-symmetric then so are $f_1$ and $f_2$.
%Moreover, $f$ can be written $gg^\dagger$ if and only if $f_1$ has no $\dagger$-symmetric irreducible factors.
%Hence, by consider all possibilities for $f_2$, the number of such $f \in \P_a^\dagger(n)$ is
%\[
%    \ll \sum_{d < n/2} q^{d/2} q^{(n-2d-1)/2} / (n-2d)^{1/2} + q^{n/4} \ll q^{(n-1)/2} / n^{1/2}.
%\]
\end{enumerate}
\end{proof}

\subsection{Some auxiliary results}

The following two propositions are of somewhat specialist interest (but will be crucial in the proof of \Cref{t_main}).
Given a polynomial $f$, let us say $f$ has property $P_r$ if every irreducible factor of $f$ has either degree or multiplicity divisible by $r$.
Thus for example $f$ has $P_2$ if and only if every odd-degree irreducible factor of $f$ has even multiplicity.

A version of part (1) of the following result appears in \cite{gorodetsky}*{Section~2.2}.

\begin{proposition}
\label{prop:p_r}
Let $r \ge 2$.
\begin{enumerate}
    \item The number of polynomials in $\P_a(n)$ ($a \in \F_q^\times$) with property $P_r$ is $\ll q^{n-1} n^{-1+1/r}$.
    \item The number of polynomials in $\P_a^*(n)$ ($a=\pm 1$) with property $P_r$ is
    \[
    \ll |\ca P^*_a(n)| \begin{cases}
        n^{-1 + 1/r} &: r~\text{odd}, \\
        n^{-1 + 3/(2r)} &: r~\text{even}.
    \end{cases}
    \]
    \item The number of polynomials in $\P_a^\dagger(n)$ ($a \in U$) with property $P_r$ is
    \[
    \ll q^{(n-1)/2} \begin{cases}
        n^{-1 + 1/r} &: r~\text{odd}, \\
        n^{-1 + 1/(2r)} &: r~\text{even}.
    \end{cases}
    \]
\end{enumerate}
\end{proposition}
\begin{proof}
\begin{enumerate}[wide]
    \item
    First, an application of \Cref{prop:poisson-tail} shows that the number of square-free $f \in \P_a(n)$ having no factors in the set $I$ of irreducible polynomials with degree not divisible by $r$ is $\ll q^{n-1} e^{-H(I)}$.
    Evidently (using \Cref{prop:PNT}(1)) $H(I) = (1-1/r) \log n + O(1)$, so the bound just quoted is $\ll q^{n-1} n^{-1+1/r}$.
    Now an arbitrary $f$ can be written uniquely as $f = f_1f_2^2f_3^r$, where $f_1 f_2^2$ is $r$-free (not divisible by any nontrivial $r$-th power) and $f_1$ is square-free,
    and clearly $f$ has property $P_r$ if and only if $f_1f_2^2$ has no factors in $I$. Thus, considering all possibilities for $f_2$ and $f_3$, which say have degrees $d$ and $e$, the number of such $f$ is bounded by
    \[
        \sum_{\substack{d, e\ge 0 \\ 2d + re < n}} q^d q^e q^{(n-2d-re) - 1} (n - 2d - re)^{-1+1/r} + \sum_{\substack{d, e \ge 0 \\ 2d + re = n}} q^d q^e
        \ll q^{n-1} n^{-1+1/r}.
    \]
    \item We consider the case $\ca P^*_1(2n)$; the general case is similar, using \eqref{eq:star-reductions}. Let $I \subset \I_{\le 2n}^*$ be the set of $*$-irreducible polynomials which are either $*$-symmetric irreducible of degree not divisible by $r$ or of the form $gg^*$ with $\deg g$ not divisible by $r$.
    Then
    \[
        H^*(I) = \sum_{r \nmid 2d\le 2n} \frac{\pi_q^*(2d)}{q^d} + \sum_{r \nmid d \le n} \frac{\pi_q^\nstar(d)/2}{q^d} + (1+\eta)q^{-1/2}.
    \]
    If $r$ is odd this is $(1-1/r) \log n + O(1)$;
    otherwise it is $(1 - 3/(2r)) \log n + O(1)$.
    The rest of the proof is as above.
    \item Similar.\qedhere
\end{enumerate}
\end{proof}

\begin{proposition}
\label{prop:technical_4Z}
    The number of polynomials in $\P_a^*(n)$ ($a=\pm 1$) with an even number of $*$-symmetric irreducible factors of degree $k$ for each $k \in [1, n/2] \cap 4 \Zint$ is
    \[
        \ll |\ca P^*_a(n)| n^{-1/4} \log n.
    \]
\end{proposition}
\begin{proof}
By \eqref{eq:star-reductions}, it is enough to prove the statement in the case of $\P_1^*(2n)$. Let $I_i \subset \I_{\le 2n}^*$ be the set of $*$-symmetric irreducible polynomials of degree $4i$ for $1 \le i \le n/2$
    and let $I_0 = \I_{\le n}^* \sm \bigcup_{1 \le i \le n/2} I_i$.
    Applying \Cref{prop:poisson-tail}(2) to this partition, we find that the number of square-free polynomials $f \in \P_1^*(2n)$ having exactly $m_i$ factors in $I_i$ for $0 \le i \le r = \floor{n/2}$ is
    \[
        \ll q^n \prod_{i=0}^r\br{e^{-H^*(I_i)} \frac{H^*(I_i)^{m_i}}{m_i!}} \br{\frac{m_0}{H^*(I_0)} + \frac{m_1}{H^*(I_1)} + \cdots + \frac{m_r}{H^*(I_r)}}.
    \]
    Summing over $m_0 \ge 0$ gives
    \[
        \ll q^n \prod_{i=1}^r \br{e^{-H^*(I_i)} \frac{H^*(I_i)^{m_i}}{m_i!}} \br{1 + \frac{m_1}{H^*(I_1)} + \cdots + \frac{m_r}{H^*(I_r)}}.
    \]
    Observe that
    \begin{align*}
        &\sum_{m~\text{even}} e^{-\lambda} \frac{\lambda^m}{m!} \phantom{\frac{m}{\lambda}} = e^{-\lambda} \cosh(\lambda),\\
        &\sum_{m~\text{even}} e^{-\lambda} \frac{\lambda^m}{m!} \frac{m}{\lambda} = e^{-\lambda} \sinh(\lambda).
    \end{align*}
    Hence the sum over all even $m_1, \dots, m_r$ is
    \[
        \ll q^n \prod_{i=1}^r f\br{H^*(I_i)} \br{1 + \sum_{i=1}^r g(H^*(I_i))}.
    \]
    where
    \begin{align*}
        f(\lambda) &= e^{-\lambda} \cosh(\lambda) = 1 - \lambda + O(\lambda^2),\\
        g(\lambda) &= \tanh(\lambda) = \lambda + O(\lambda^2).
    \end{align*}
    Now, by \Cref{prop:PNT}(2),
    \[
        H^*(I_i) = \frac{\pi_q^*(4i)}{q^{2i}} = \frac{1}{4i} \br{1 - O(q^{-4i/3})},
    \]
    so we get the claimed bound $\ll q^n n^{-1/4} \log n$, in the square-free case.

    As usual, to deduce the arbitrary case we write an arbitrary polynomial $f$ as $f = f_1f_2^2$ with $f_1$ square-free, and we note that $f$ has an even number of irreducible factors of degree $k$ if and only if $f_1$ does. Hence the number of $f \in \P_1^*(2n)$ with an even number of $*$-symmetric irreducible factors of degree $k$ for each $k \in [1,n] \cap 4\Zint$ is
    \[
        \ll \sum_{\substack{\deg f_2 < n \\ f_2^* = f_2}} q^{n - \deg f_2} (n-2\deg f_2)^{-1/4} \log (n - 2 \deg f_2) + q^{n/4} \ll q^n n^{-1/4} \log n.\qedhere
    \]
\end{proof}

\subsection{Parity of the number of irreducible factors}

For $f \in \F_q[X]$, let  $\Omega(f)$ be the number of irreducible factors counting multiplicity and let $\lambda(f) = (-1)^{\Omega(f)}$.
Define $\mu(f) = \lambda(f)$ when $f$ is square-free and $\mu(f) = 0$ otherwise.
These are the analogues of the Liouville and M\"obius functions for integers.

In the case of unrestricted polynomials, it is straightforward to establish the generating function identities
\begin{align*}
    \sum_{f \in \P} z^{\deg f} &= (1-qz)^{-1}, \\
    \sum_{f \in \P} \mu(f) z^{\deg f} &= 1 - qz, \\
    \sum_{f \in \P} \lambda(f) z^{\deg f} &= (1 - qz) / (1 - qz^2).
\end{align*}
By taking the coefficient of $z^n$ we get the identities $\sum_{f \in \P(n)} \mu(f) = 0$ for $n > 1$ and $\sum_{f \in \P(n)} \lambda(f) = (-1)^n q^{\ceil{n/2}}$.
We need a variant for $*$-symmetric polynomials.

\begin{proposition}
\label{prop:parity_factors}
    Let $\calQ(n) \subset \P_1^*(n)$ be the set of $*$-symmetric polynomials $f$ such that $f(1), f(-1) \ne 0$. Then
    \[
        \sum_{f \in \calQ(n)} \lambda(f) = \begin{cases}
            1 &: n = 0, \\
            -1 &: n = 2, \\
            0 &: n \ne 0, 2.
        \end{cases}
    \]
    Hence
    \[
        \sum_{f \in \P^*(n)} \lambda(f) = \begin{cases}
            1 &: n = 0,\\
            -1 &: n = 1 ~\text{and}~ 2 \mid q,\\
            0 &: n > 1 ~\text{and}~ 2 \mid q,\\
            2 (-1)^n &: n \ge 1 ~\text{and}~ 2 \nmid q.
        \end{cases}
    \]
\end{proposition}
\begin{proof}
Let $\calQ = \bigcup_{n \ge 0} \calQ(n)$ and let $F(z) = \sum_{f \in \calQ} \lambda(f) z^{\deg f}$.
Then we have an Euler product expression
\[
    F(z) = \prod_{\substack{g~*\text{-irred.} \\ g \ne X \pm 1}} (1 - \lambda(g) z^{\deg g})^{-1}
    = \prod_{n \ge 1} (1 + z^{2n})^{-\pi_q^*(2n)} (1 - z^{2n})^{-\pi_q^\nstar(n)/2}.
\]
Let $\eta = \ind{q~\text{odd}}$.
From \Cref{prop:PNT}(2),
\begin{align*}
    \pi_q^*(2n) &= \frac{1}{2n} \sum_{\text{odd}~d \mid n} \mu(d) (q^{n/d} - \eta), \\
    \pi_q^\nstar(n)/2 &= \frac{1}{2n} \br{\sum_{d \mid n} \mu(d) q^{n/d} - \sum_{\text{odd}~d \mid n/2} \mu(d) (q^{n/2d} - \eta) - (2 + \eta)\ind{n=1}} \\
    &= \frac{1}{2n} \br{\sum_{d \mid n} \mu(d) (q^{n/d} - \eta) - \sum_{\text{odd}~d \mid n/2} \mu(d) (q^{n/2d} - \eta) - 2\ind{n=1}} \\
    &= \frac{1}{2n} \br{\sum_{\text{odd}~d \mid n} \mu(d) (q^{n/d} - \eta) - 2 \sum_{\text{odd}~d \mid n/2} \mu(d) (q^{n/2d} - \eta) - 2\ind{n=1}}.
\end{align*}
Hence
\begin{align*}
    \log F(z)
    &= \sum_{n=1}^\infty \sum_{m=1}^\infty \br{\br{(-1)^m+1} \sum_{\text{odd}~d \mid n} \mu(d) (q^{n/d} - \eta)
    - 2\sum_{\text{odd}~d \mid n/2} \mu(d) (q^{n/2d} - \eta)} \frac{z^{2mn}}{2mn} \\
    &\hspace{2cm} + \log (1 - z^2).
\end{align*}
The coefficient of $z^N / N$ in the sum is
\[
    \sum_{n \mid N/2} \sum_{\text{odd}~d \mid n} \mu(d) (q^{n/d} - \eta) - \sum_{n \mid N} \sum_{\text{odd}~d \mid n/2} \mu(d) (q^{n/2d} - \eta) = 0.
\]
Hence\footnote{Is there a proof which is a little less humpty-dumpty?} $F(z) = 1- z^2$. Now taking the coefficient of $z^n$ gives the formula for $\sum_{f \in \calQ(n)} \lambda(f)$. The second formula follows similarly from
\[
    \sum_{f \in \P^*} \lambda(f) z^{\deg f} = \prod_{\substack{g~*\text{-irred}}} (1 - \lambda(g) z^{\deg g})^{-1} = (1 + z)^{-1 - \eta} F(z). \qedhere
\]
\end{proof}

\begin{remark}
    Let $q$ be odd and let $f \in \calQ(n)$. By a result of Ahmadi and Vega~\cite{ahmadi--vega}*{Theorem~12}, $\lambda(f)$ is $+1$ if and only if $(-1)^n f(1) f(-1)$ is a square in $\F_q$.
    This result can be used to give an alternate proof of \Cref{prop:parity_factors} in the odd-characteristic case.
\end{remark}

\subsection{Polynomials without low-degree factors}

The proof of the next proposition is loosely inspired by the Brun--Hooley sieve from analytic number theory.

\begin{proposition}
\label{prop:bounded_degree}
Let $1 \le k \le n / (10\log n)$.
\begin{enumerate}
\item Let $f \in \P_a(n)$ ($a \in \F_q^\times$) be uniformly random, let $I_k$ be the set of irreducible polynomials of degree $\le k$, and let $E_k$ be the event that $f$ has no factors in $I_k$. Then
\[\Prob(E_k) = \prod_{g \in I_k} (1 - q^{-\deg g}) + O(e^{-cn / k}).\]
\item Let $f \in \P^*_a(n)$ ($a = \pm 1$) be uniformly random, let $I_k$ be the set of $*$-irreducible polynomials of degree $\le k$, other than $X \pm 1$, and let $E_k$ be the event that $f$ has no factors in $I_k$. Then
\[\Prob(E_k) = \prod_{g \in I_k} (1 - q^{-(\deg g)/2}) + O(e^{-cn / k}).\]
Moreover, the same estimate holds if $n$ is even, $a = 1$, and $f$ is conditioned to have an even or odd number of irreducible factors and no factors of $X \pm 1$.
\item Let $f \in \P^\dagger_a(n)$ ($a \in U$) be uniformly random, let $I_k$ be the set of $\dagger$-irreducible polynomials of degree $\le k$, and let $E_k$ be the event that $f$ has no factors in $I_k$. Then
\[\Prob(E_k) = \prod_{g \in I_k} (1 - q^{-(\deg g)/2}) + O(e^{-cn / k}).\]
\end{enumerate}
\end{proposition}
\begin{proof}
\begin{enumerate}[wide]
\item Recall the classical Bonferroni inequalities (truncated inclusion--exclusion), which state that, for any collection of events $\mathcal{E}$,
\begin{align*}
    \Prob\br{\bigcup\mathcal{E}} &\le \sum_{i=1}^m (-1)^{m-1} \sum_{\substack{\mathcal{A} \subset \mathcal{E} \\ |\mathcal{A}| = i}} \Prob\br{\bigcap \mathcal{A}}&& (m~\text{odd}), \\
    \Prob\br{\bigcup\mathcal{E}} &\ge \sum_{i=1}^m (-1)^{m-1} \sum_{\substack{\mathcal{A} \subset \mathcal{E} \\ |\mathcal{A}| = i}} \Prob\br{\bigcap \mathcal{A}}&& (m~\text{even}).
\end{align*}
We apply these with $\mathcal{E} = \{\ind{g \mid f} : g \in I_k\}$ and $m = \floor{(n-1) / k}$. Note that $E_k$ is the complement of $\bigcup \mathcal{E}$. Provided $\deg g \le n-1$, the probability that $g \mid f$ is exactly $q^{-\deg g}$. Thus since $mk \le n-1$ we have
\[
    \sum_{\substack{\mathcal{A} \subset \mathcal{E} \\ |\mathcal{A}| = i}} \Prob\br{\bigcap \mathcal{A}} = \sum_{\substack{A \subset I_k \\ |A| = i}} \prod_{g \in A} q^{-\deg g}
\]
for all $i \le m$. It follows that
\[
    \left|\Prob(E_k) - \prod_{g \in I_k} (1 - q^{-\deg g})\right|
    \le \sum_{i \ge m} \sum_{\substack{A \subset I_k \\ |A| = i}} \prod_{g \in A} q^{-\deg g}
    < \sum_{i \ge m} \frac{H(I_k)^i}{i!}.
\]
Since $m > 3 \log k$ and $H(I_k) \le H_k \le 1 + \log k$, the sum above is
\[
    \ll H(I_k)^m / m! \ll (e H(I_k) / m)^m \le (C k \log k / n)^{-cn/k}.
\]

\item Similar. The only point to emphasize is that, since we have excluded $X \pm 1$, the product $g$ of any subset of $I_k$ is a $*$-symmetric polynomial of even degree with constant coefficient $+1$, so we have $\Prob(g \mid f) = q^{-(\deg g)/2}$ provided $\deg g < n$.
Alternatively one can use \eqref{eq:star-symmetric-reduction} and \eqref{eq:star-reductions}. We omit the details.

Now consider the case in which $f$ is conditioned to have no factors of $X \pm 1$ and an even or odd number of irreducible factors.
By the formulae \eqref{eq:P_a-size} and inclusion--exclusion, the number of $f \in \P_1^*(n)$ with no factors of $X \pm 1$ is $q^{n/2} - 2q^{n/2 - 1} + q^{n/2 - 2} = q^{n/2} (1 - q^{-1})^2$ if $q$ is odd and $q^{n/2} - q^{n/2 - 1} = q^{n/2} (1 - q^{-1})$ if $q$ is even. Applying \Cref{prop:parity_factors}, the number of $f$ with an even (or odd) number irreducible factors is exactly half that, provided $n > 2$. Hence if $g$ is the product of any subset of $I_k$ such that $\deg g \le n - 4$ then again $\Prob(g \mid f) = q^{-(\deg g) / 2}$. Hence we can repeat the argument of (1) using $m = \floor{(n-4) / k}$.

\item Similar.
\qedhere
\end{enumerate}
\end{proof}

\section{Preliminaries for groups of Lie type}
\label{sec:preliminaries}
In this section we prove some preliminary results on finite groups of Lie type. Since some readers may be more comfortable with polynomials than with groups, we give more details than one may normally do.

\subsection{Definitions}

Let us agree on the definitions of the finite classical groups.
Refer to one of the many books on the subject for more details (e.g., \cite{kleidman_liebeck}*{Chapter~2} or \cite{aschbacher-book}*{Chapter~7} or \cite{atlas}*{Chapter~2}).
Let $q$ be a power of a prime $p$.

\begin{itemize}[label=$\diamond$]
    \item $\GL_n(q)$ is the group of linear automorphisms of the finite vector space $\F_q^n$, whose members we can freely identify with matrices of nonzero determinant.
    \begin{itemize}[label=$\circ$]
        \item $\SL_n(q)$ is the subgroup of $\GL_n(q)$ consisting of matrices with determinant $1$.
    \end{itemize}
    \item $\GU_n(q)$ is the isometry group of a nondegenerate unitary form on $\F_{q^2}^n$.
    \begin{itemize}[label=$\circ$]
        \item $\SU_n(q)=\GU_n(q) \cap \SL_n(q^2)$.
    \end{itemize}
    \item $\Sp_{2n}(q)$ is the isometry group of a nondegenerate alternating form on $\F_q^{2n}$.
    \item $\Or^\vareps_n(q)$ is the isometry group of a nondegenerate quadratic form on $\F_q^n$. If $n$ is even then $\vareps \in \{+, -\}$ indicates the type of the quadratic form (the Witt defect is $(1- \vareps)/2$). If $n$ is odd then $q$ must be odd and $\vareps$ may be omitted, or we may write $\vareps = \circ$ according to notational convenience.
    \begin{itemize}[label=$\circ$]
        \item $\SO^\vareps_n(q) = \Or^\vareps_n(q) \cap \SL_n(q)$.
        \item $\Omega^\vareps_n(q) = \SO^\vareps_n(q)'$, which for $n\ge 5$ is the unique subgroup of $\SO^\vareps_n(q)$ of index $2$.
    \end{itemize}
\end{itemize}
We call the cases respectively \emph{linear}, \emph{unitary}, \emph{symplectic}, \emph{orthogonal}.
We will often elide the linear and unitary cases by writing $\SL_n(q) = \SL_n^+(q)$ and $\SU_n(q) = \SL_n^-(q)$, and similarly for $\GL$ and $\GU$.
In each case the corresponding projective group $G / \Z(G)$ is indicated by attaching the prefix $\opr{P}$, as in $\PGL$.

To avoid trivialities or repetitions we may assume $n \ge 2$ for $\GL_n(q)$, $n \ge 3$ for $\GU_n(q)$, $n \ge 2$ for $\Sp_{2n}(q)$, and $n \ge 7$ for $\Or_n^\vareps(q)$.
With these restrictions, the quotient groups
\begin{equation}
    \label{eq_simple}
    \PSL_n^\pm(q), \PSp_{2n}(q), \POm_n^\vareps(q)
\end{equation}
are all simple, except for $\PSL_2(2)$, $\PSL_2(3)$, $\PSU_3(2)$, and $\PSp_4(2)$.
These are the \emph{(finite) simple classical groups}.
We will refer to groups $G$ such that $G' = G$ and $G/\Z(G)$ is a simple classical group as \emph{quasisimple classical groups},
and groups $G$ such that $S \le G \le \Aut(S)$ for some simple classical group $S$ as \emph{almost simple classical groups}.

Our proof will mostly refer only to the classical quasisimple groups
\begin{align*}
    &G= \SL^\pm_n(q), \Sp_{2n}(q), \Omega^\varepsilon_{n}(q)  \\
    &\quad\text{with $G$ quasisimple}, \nonumber
\end{align*}
or more generally
\begin{align}
\label{eq:almost-quasisimple}
    &\SL^\pm_n(q)\leq G \leq \GL^\pm_n(q), G=\Sp_{2n}(q), \Omega^\varepsilon_{n}(q) \\
    &\quad \text{with $G'$ quasisimple}. \nonumber
\end{align}

Each of the groups $G$ in \eqref{eq:almost-quasisimple} is defined as a subgroup of $\GL_m(q)$ or $\GL_m(q^2)$ for some $m$. As usual an element $g\in G$ is called \emph{semisimple} if it is diagonalizable over $\overline{\F_q}$.
We call a conjugacy class semisimple if it consists of semisimple elements.
For finite groups, as here, semisemplicity is equivalent to having $p'$-order.

The simple classical groups make up the bulk of the finite simple groups of Lie type.
The remaining groups are called the \emph{exceptional groups of Lie type}.
They are best viewed through the lens of algebraic groups, which we now review. A complete treatment can be found for example in \cite{malle_testerman}.
Even in the case of classical groups we will find the perspective of algebraic groups useful on some occasions (particularly when it comes to Shintani descent).

If $X$ is a simple linear algebraic group over $\overline{\F_p}$, with Steinberg endomorphism $\sigma$, we write $X_\sigma=\{x\in X \mid x^\sigma = x\}$.
We require that $(X_\sigma)'$ is perfect, which holds in all but a handful of cases.
If $X$ is of adjoint type then $S = (X_\sigma)'$ is a finite simple group, \emph{a finite simple group of Lie type}.
The \emph{untwisted rank} of $S$ is the rank of $X$.

We canonically associate a parameter $q$ to the pair $(X,\sigma)$ as follows. Let $T$ be a $\sigma$-stable maximal torus of $X$, so $\sigma$ acts naturally on the character group $\Hom(T,\mathbf G_m)$. Then, the eigenvalues of $\sigma$ on $\Hom(T,\mathbf G_m)\otimes_{\mathbf Z} \mathbf C$ all have the same absolute value, which we denote by $q$, and which is a fractional power of $p$ (see \cite{malle_testerman}*{Lemma~22.1 and Proposition~22.2}). We will refer to $q$ as either the \emph{level} of $(X,\sigma)$, or the level of $X_\sigma$, or the level of $\sigma$ if $X$ is understood.
For classical groups $q$ is the parameter appearing in \eqref{eq_simple}.
For exceptional groups which are not Suzuki or Ree, $q$ is an integer and $\F_q$ can similarly be thought of as the field of definition.
For Suzuki and Ree groups, $q$ is not an integer but $q^2$ is an integer.

Assume now that $X$ is simple and of adjoint type, with Steinberg endomorphism $\sigma$, so that $(X_\sigma)'=S$ is a finite simple group of Lie type. The group of \emph{inner-diagonal automorphisms} is
\begin{equation*}
%    \label{eq:Inndiag}
    \Inndiag(S) = X_\sigma,
\end{equation*}
so $\Inndiag(S) \leq \Aut(S)$.
Concretely, for $S=\PSL^\pm_n(q)$, $\Inndiag(S)=\PGL^\pm_n(q)$, and in all other cases $|\Inndiag(S):S|\leq 4$.
Moreover if $S$ has level $q$ then $|\Aut(S)/\Inndiag(S)|\ll \log q$.
See for instance \cite{gorenstein}*{Theorem~2.5.12} for the precise structure of $\Inndiag(S)/S$.

\subsection{Basic results on $k(G)$ and semisimple classes}

In this subsection we collect some general results on conjugacy classes, particularly their number $k(G)$.
The first is a basic general relation from \cite{gallagher} between $k(G)$ and $k(H)$ when $[G:H]$ is bounded.

\begin{lemma}
[\cite{gallagher}]
%\label{lem:k-subgroup}
If $G$ is a finite group and $H$ is a subgroup of $G$, then
\[
    |G:H|^{-1} k(H)
    \le
    k(G)
    \le |G:H|  k(H).
\]
\end{lemma}

\begin{lemma}
    \label{lem:passing-to-a-subgroup}
    Let $G$ be a group acting transitively on a finite set $\Omega$ and let $H \nsgp G$.
    Then
    \[
        \derang(G, \Omega) \ge \derang(H, \Omega) / |G:H|^2.
    \]
\end{lemma}
\begin{proof}
    Every $G$-conjugacy class contained in $H$ splits into at most $|G:H|$ $H$-conjugacy classes.
    It follows that the number of $G$-conjuacy classes containing derangements is at least $k(H) \derang(H, \Omega) / |G:H|$.
    Hence the result follows from the previous lemma.
\end{proof}

Next we need several results from \cite{fulman2012_conjugacy_classes}.
The first gives a bound for $k(G)$ for groups of Lie type,
and also a bound for the number of non-semisimple classes.

\begin{theorem}
[\cite{fulman2012_conjugacy_classes}*{Theorem 1.1}]
\label{t:fulman_guralnick_bound}
    Let $X$ be a simple linear algebraic group of rank $r$ over $\overline{\F_p}$, and let $\sigma$ be a Steinberg endomorphism of $X$ of level $q$. Then
    \[
        k(X_\sigma) = q^r + O(q^{r-1}).
    \]
    Moreover the number of non-semisimple conjugacy classes is $O(q^{r-1})$.
\end{theorem}

\begin{lemma}
[\cite{fulman2012_conjugacy_classes}*{Corollaries~3.7 and 3.11}]
\label{lem:k-for-GL-and-PGL}
\leavevmode
\begin{enumerate}
    \item Assume $\SL_n^\pm(q) \le G \le \GL_n^\pm(q)$ and $t = |\GL_n^\pm(q) : G|$. Then
    \[
        k(G) \asymp q^n / t.
    \]
    \item Assume $\PSL_n^\pm(q) \le G \le \PGL_n^\pm(q)$ and $t = |\PGL_n^\pm(q) : G|$. Then
    \[
        k(G) \asymp q^{n-1} / t.
    \]
\end{enumerate}
\end{lemma}

%When we use the following lemma we will usually take $\pi$ to be the set of all primes, but it is also occasionally useful to take $\pi = \{p\}^c$.

The following lemma is a special case of \cite{fulman2012_conjugacy_classes}*{Lemma~2.2}. In the statement, a conjugacy class $C$ of $N$ is called $a$-stable if $C^a=C$.
\begin{lemma}
    \label{l_shintani_analogue}
    Let $G$ be a finite group and let $N$ be a normal subgroup of $G$ with $G=\gen{N,a}$ and $a\in G$.
%    For any $b \in G$, the number of $G$-stable $N$-classes contained in $Nb$ is a constant independent of $b$.
%    In particular,
    The number of $N$-classes in $Na$ is equal to the number of $a$-stable conjugacy classes in $N$.
\end{lemma}

The previous result is similar to a tool known as \emph{Shintani descent} for algebraic groups,
which is an essential tool for understanding the conjugacy classes of almost simple groups.
The following general version is from \cite{harper2021shintani} (see \cite{harper2021shintani}*{Theorem~2.1 and Remark~2.3}).

\begin{lemma}
[Shintani descent]
\label{l_shintani_descent}
Let $X$ be a connected linear algebraic group over $\overline{\F_p}$, and let $\sigma_1$ and $\sigma_2$ be commuting Steinberg endomorphisms of $X$.
Consider the cosets
\[
    X_{\sigma_1} \sigma_2 \subset X_{\sigma_1} \rtimes \gen{\sigma_2}
    ~ \text{and}~
    X_{\sigma_2} \sigma_1 \subset X_{\sigma_2} \rtimes \gen{\sigma_1}.
\]
There is a bijection
\[
	\{\textup{$X_{\sigma_1}$-classes in $X_{\sigma_1}\sigma_2$ }\}
	\longleftrightarrow
	\{\textup{$X_{\sigma_2}$-classes in $X_{\sigma_2}\sigma_1$}\}.
\]
\end{lemma}

For example, if $X=\GL_n(\overline{\F_p})$, $\sigma_1$ maps each matrix entry to its $q$-th power, and $\sigma_2=\sigma_1^b$, then $X_{\sigma_1} = \GL_n(q)$ and $X_{\sigma_2} = \GL_n(q^b)$.
In this case \Cref{l_shintani_descent} asserts that the number of $\GL_n(q^b)$-classes in the coset $\GL_n(q^b)\sigma_1$ is equal to the number of conjugacy classes of $\GL_n(q)$.

\subsection{Maximal subgroups of classical groups}
%\label{subsec:maximal_subgroups}
Aschbacher~\cite{aschbacher1984maximal} gave a description of the maximal subgroups of the almost simple classical groups in terms of the (projective) action of $G$ on its defining module.

Aschbacher divided the maximal subgroups not containing the socle into nine classes, which can be roughly described as follows; here $V$ denotes the natural module for the group.

\begin{itemize}
    \item[$(\ca C_1)$] Stabilizers of certain subspaces of $V$.
    \item[$(\ca C_2)$] Stabilizers of direct sum decompositions $V=V_1\oplus \cdots \oplus V_t$.
    \item[$(\ca C_3)$] Extension field subgroups.
    \item[$(\ca C_4)$] Stabilizers of tensor product decompositions $V=V_1 \otimes V_2$.
    \item[$(\ca C_5)$] Subfield subgroups.
    \item[$(\ca C_6)$] Symplectic-type subgroups.
    \item[$(\ca C_7)$] Stabilizers of tensor product decompositions $V=V_1\otimes \cdots \otimes V_t$.
    \item[$(\ca C_8)$] Classical subgroups in natural action.
    \item[$(\ca S)$] Almost simple groups acting absolutely irreducibly, and not belonging to the previous classes.
\end{itemize}

% \textbf{Class $\ca C_1$:} Assume first $\SL_n(q) \leq G \leq \GL_n(q)$. Then, for each $1\leq k < n$, $k\neq n/2$, there is a conjugacy class of maximal subgroups of $G$ of class $\ca C_1$, consisting of stabilizers of $k$-dimensional subspaces.

% Assume now $G$

Recall that, when $q$ is even, $\Sp_{2n}(q)$ can be identified with $\SO_{2n+1}(q)$ -- the group of isometries of a nonsingular
quadratic form $Q$ on $V=\F_q^{2n+1}$ (here nonsingular means that $V^\perp$ is an anisotropic subspace).
Now, $\Sp_{2n}(q)$ contains maximal subgroups $\SO^+_{2n}(q)$ and $\SO^-_{2n}(q)$, which are usually placed in class $\ca C_8$, as in \cite{kleidman_liebeck} for example.
However, under the identification $\Sp_{2n}(q)\cong\SO_{2n+1}(q)$, these subgroups correspond to stabilizers of nondegenerate hyperplanes of plus and minus type, respectively. The following alternative convention is therefore reasonable and adopted in this paper.

\begin{convention}
\label{convention_c1}
When $q$ is even, the maximal subgroups $\SO^+_{2n}(q)$ and $\SO^-_{2n}(q)$ of $\Sp_{2n}(q)$ belong to class $\ca C_1$.
\end{convention}

This convention is implicit in \cite{fulman2012_conjugacy_classes}*{Theorem~1.3}.

\subsection{Semisimple classes and polynomials}
%\label{subsec:semisimple_classes}

We review the well-known correspondence between conjugacy classes of semisimple elements of classical groups and polynomials. Much of our discussion follows \cite{fulman2013regular_semisimple}.

In non-orthogonal groups, the semisimple classes are in one-to-one correspondence with suitable sets of polynomials. For the reader's convenience, we give a proof of this fact, using standard tools from the theory of algebraic groups.

\begin{lemma}
\label{l_semisimple_classes_characteristic_polynomial}
Let $G$ be one of $\GL_n(q)$, $\GU_n(q)$, $\Sp_{2n}(q)$ and assume $G'$ is quasisimple.
Any two semisimple elements of $G$ are $G'$-conjugate if and only if they have the same characteristic polynomial.
\end{lemma}

\begin{proof}
The ``only if'' part is clear, so we focus on the ``if'' part.

The linear and unitary cases can be treated uniformly, as follows.
Let $K = \overline{\F_q}$ and put $H=\GL_n(K)$ and $Z = \Z(H) \cong K^*$, so that $H=H' Z$ with $H' = \SL_n(K)$.
In particular, for a semisimple element $s=s'z$ of $H$, with $s'\in H'$ and $z\in Z$, we have $\C_{H'}(s) = \C_{H'}(s')$. By a theorem of Steinberg \cite{malle_testerman}*{Theorem 14.16}, it follows that $\C_{H'}(s)$ is connected. Now, for every Steinberg endomorphism $\sigma$ of $H$, the Lang--Steinberg theorem \cite{malle_testerman}*{Theorem~21.11} implies that, if $s\in H_\sigma$ is semisimple, then $(H')_\sigma$ acts transitively by conjugation on $(s^{H'})_\sigma = (s^H)_\sigma$. Now, two semisimple elements of $H$ with the same characteristic polynomial are conjugate in $H$. We can choose $\sigma$ so that $H_\sigma = \GL_n^\pm(q)$, which proves the statement for these two groups.

Assume now $H=H'=\Sp_{2n}(K)$. By the same argument as above, we have that for every semisimple element $s\in H_\sigma = \Sp_{2n}(q)$, $\Sp_{2n}(q)$ acts transitively by conjugation on $(s^H)_\sigma$. In particular, in order to prove the statement it is enough to show that two semisimple elements $s$ and $t$ of $H$ with the same characteristic polynomial are conjugate in $H$.

The space $K^{2n}$ splits as a direct sum $V=W\oplus W'$ of two totally singular spaces invariant under $s$, and similarly for $t$, say $V=U\oplus U'$. We may choose $U$ so that the characteristic polynomial of $s$ on $W$ is equal to the characteristic polynomial of $t$ on $U$. Since $\Sp_{2n}(K)$ acts transitively on pairs of complementary maximal totally singular spaces, we may conjugate $t$ and assume that $W=U$, $W'=U'$. Since $\Sp_{2n}(K)$ contains a subgroup $\GL(W)$ stabilizing the decomposition, we may then conjugate $s$ to $t$. This concludes the proof.
\end{proof}

Next we consider orthogonal groups. We use the following notation:
\[
    O=\Or^\varepsilon_n(q),\quad S=\SO^\varepsilon_n(q),\quad \Omega=\Omega^\varepsilon_n(q),
\]
where $\varepsilon \in\{+,-\}$ if $n$ is even and $\varepsilon = \circ$ if $n$ is odd.

For orthogonal groups in odd characteristic, two elements of $O$ with the same characteristic polynomial need not be conjugate. We now specify when this happens.

We define $\ca M(n)$ to be a subset of the polynomials $f \in \ca P^*_{(-1)^n}(n)$ with some additional data. To be precise, let
\[
    \ca M(n) = \ca M_0(n) \cup \ca M_1(n) \cup \ca M_2(n),
\]
where
\begin{enumerate}[start=0]
    \item $\ca M_0(n)$ is the set of polynomials $f \in \ca P_{(-1)^n}^*(n)$ without $\pm1$ as a root and such that the number of $*$-symmetric irreducible factors of $f$ is even if $\vareps = +$ and odd if $\vareps = -$,
    \item $\ca M_1(n)$ is the set of polynomials $f$ with exactly one root in $\{\pm1\}$ (ignoring multiplicity),
    \item $\ca M_2(n)$ is the set of polynomials $f$ with two roots in $\{\pm1\}$, together with an additional datum $\xi \in \{+, -\}$.
\end{enumerate}
Note that $\ca M_0(n) = \emptyset$ if $n$ is odd and $\ca M_2(n) = \emptyset$ if $q$ is even.
Consider now the map
\[
    \Phi\colon \text{\{semisimple $O$-classes contained in $S$\}} \to \ca M(n),
\]
given by mapping a class $C$ to its characteristic polynomial and, if $q$ is odd and the elements of $C$ have both $1$ and $-1$ as eigenvalues, the type $\xi$ of the $1$-eigenspace (which is always nondegenerate). Here we use the convention that a space of odd dimension has plus type if it has square discriminant.

\begin{fact}
\label{fact}
$\Phi$ is a bijection.
\end{fact}
\begin{proof}
    See \cite{fulman2013regular_semisimple}*{Lemma~5.2 and Lemma~5.8} (note the authors work with regular classes but the proof is valid in general).
\end{proof}

If $q$ is even, then every semisimple class of $S$ is contained in $\Omega$, since $|S:\Omega|=2$ and semisimple elements have odd order. If $q$ is odd, however, this is not true, and some care is needed in order to identify semisimple classes in $\Omega$. We now handle this issue.

\begin{lemma}
	[\cite{kleidman_liebeck}*{Proposition~2.5.13}]
\label{l_conditions_-1}
Assume that $q$ is odd and $n\equiv 0\pmod 4$. Let $\Omega = \Omega^-_{n}(q)$. Then $-1\in S \sm \Omega$ and $S = \Omega \times \langle -1 \rangle.$
% one of the following hold:
% \begin{itemize}
%     \item[(i)] $G=\Omega^+_{n}(q)$, and $q\equiv 3\pmod 4$ and $n\equiv 2\pmod 4$.
%     \item[(ii)] $G=\Omega^-_{n}(q)$, and $q\equiv 1\pmod 4$ or $n\equiv 0\pmod 4$.
% \end{itemize}
% Then, $-1\in S\sm G$.
\end{lemma}

In the following lemma, for a normal subset $Y$ of $O$, we denote by $\ca C(Y)$ the set of $O$-classes contained in $Y$.

\begin{lemma}
\label{l_equidistribution_cosets}
Assume that $q$ is odd. Let $\ca A(n)$ be the subset of $\ca M(n)$ consisting of those elements whose underlying polynomial $f$ satisfies ($\star_k$) for some $k\in [1,n/2] \cap 4\mathbf Z$, where ($\star_k$) is the following condition:
\begin{itemize}
\item[($\star_k$)] $f$ has an odd number of $*$-symmetric irreducible factors of degree $k$.
\end{itemize}
Then
\begin{equation}
	\label{eq:An-even-split}
    |\Phi^{-1}(\ca A(n)) \cap \ca C(\Omega)| = |\Phi^{-1}(\ca A(n)) \cap \ca C(S\sm \Omega)|.
\end{equation}
Moreover, the same is true if, in both sides of the above equality, we further intersect with the classes having nonempty intersection with any fixed maximal subgroup of $S$ of class $\ca C_1$.
\end{lemma}

\begin{proof}
Write
\[
	\ca A(n) =\bigcup_{k\in [1,n/2] \cap 4\mathbf Z}\ca A_k,
\]
where $\ca A_k$ denotes the subset of $\ca A(n)$ satisfying ($\star_k$), but not ($\star_j$) for $j \in [1,k) \cap 4\mathbf Z$.
It suffices to prove \eqref{eq:An-even-split} with $\ca A_k$ in place of $\ca A(n)$.

Let $g$ be a semisimple element of $S$ such that $\Phi(g^O) \in \ca A_k$.
Let $f$ be the characteristic polynomial of $g$. Let $W$ be the sum of the $g$-invariant $k$-dimensional subspaces of $V = \F_q^n$ corresponding to $*$-symmetric irreducible factors of $f$ of degree $k$.
Note that $W$ is nondegenerate, so $V = W \oplus W^\perp$.
By \Cref{l_conditions_-1}, $-1\in \SO(W)\sm \Omega(W)$.
Let $h \in S \sm \Omega$ be the element respecting the decomposition $V = W \oplus W^\perp$ and acting as $-1$ on $W$ and $1$ on $W^\perp$.
Then $gh \in S \sm \Omega g$ and $\Phi((gh)^O) \in \ca A_k$.
It is clear that the map $g^O \mapsto (gh)^O$ is a well-defined bijection of $\Phi^{-1}(\ca A_k)$,
and this proves the claim.
(On the level of $\ca M(n)$, the bijection is defined by replacing $X$ with $-X$ in each of the $*$-symmetric irreducble factors of degree $k$.)

The last statement of the lemma follows from the observation that $g$ and $gh$ have the same invariant subspaces.
\end{proof}

We finally collect the information that we need. Below we denote the multiplicative order of an element $a \in \F_q^\times$ by $|a|$.

\begin{enumerate}[(i)]
    \item Let $\SL_n(q) \leq G\leq \GL_n(q)$, with $|G:\SL_n(q)| = t$. By \Cref{l_semisimple_classes_characteristic_polynomial} we have a bijection
% \begin{align*}
%     &\SL_n(q) \leq G\leq \GL_n(q) &\Phi\colon \text{\{semisimple classes of $G$\}} \to \bigcup_{\substack{a\in \F_q^\times \\ |a| \bigm\vert |G:\SL_n(q)|}}\ca P_a(n) \\
%     &\SU_n(q) \leq G\leq \GU_n(q) &\Phi\colon \text{\{semisimple classes of $G$\}} \to \bigcup_{\substack{a\in \F_{q^2}^\times \\ |a| \bigm\vert |G:\SU_n(q)|}}\ca P^\dagger_a(n) \\
%     &G=\Sp_{2n}(q) &\Phi\colon \text{\{semisimple classes of $G$\}} \to \ca P^*_1(2n)
% \end{align*}
\begin{equation}
\label{eq_linear}
\Phi\colon \text{\{semisimple classes of $G$\}} \to \bigcup_{\substack{a\in \F_q^\times \\ |a| \mid t}}\ca P_{(-1)^na}(n),
\end{equation}
given by associating to each class the characteristic polynomial of its elements.

\item Let $\SU_n(q) \leq G\leq \GU_n(q)$, with $|G:\SU_n(q)| = t$. Then the characteristic polynomial of the elements of $G$ belongs to $\ca P^\dagger_{(-1)^na}(n)$, with $a \in \F_{q^2}^\times, |a|\mid t$. Hence by \Cref{l_semisimple_classes_characteristic_polynomial} we have a bijection
\begin{equation}
\label{eq_unitary}
\Phi\colon \text{\{semisimple classes of $G$\}} \to \bigcup_{\substack{a\in \F_{q^2}^\times \\ |a| \mid t}}\ca P^\dagger_{(-1)^na}(n).
\end{equation}
In the unitary case we emphasize that our polynomials have coefficients in $\F_{q^2}$ (so the results of \Cref{sec:anatomy} should be applied with $q^2$).

\item Let $G=\Sp_{2n}(q)$. Then the characteristic polynomial of the elements of $G$ belongs to $\ca P^*_1(2n)$. Hence by \Cref{l_semisimple_classes_characteristic_polynomial} we have a bijection
\begin{equation}
\label{eq_symplectic}
\Phi\colon \text{\{semisimple classes of $G$\}} \to \ca P^*_1(2n).
\end{equation}

\item Let $S=\SO^\varepsilon_{n}(q)$, $O=\Or^\varepsilon_n(q)$. For $n$ even, define $\ca B(n)$ as the set of semisimple $O$-classes contained in $S$ without $\pm 1$ as eigenvalue. For $n$ odd define $\ca B(n)$ as the set of semisimple $O$-classes contained in $S$ with an eigenvalue $1$ of multiplicity one and without eigenvalue $-1$. By \Cref{fact}, for $n$ even we have a bijection
\begin{equation}
\label{eq_orthogonal_even}
\Phi\colon \ca B(n) \to \ca M_0(n),
\end{equation}
and for $n$ odd we have a bijection
\begin{equation}
\label{eq_orthogonal_odd}
\Phi\colon \ca B(n) \to \ca N(n),
\end{equation}
where $\ca N(n)$ is the set of polynomials of the form $(X-1)f$, where $f\in \ca P^*_1(n-1)$ does not have $\pm 1$ as a root.
\end{enumerate}

In the orthogonal case, the following calculation will allow us to restrict attention to $\ca M_0(n)$ or $\ca N(n)$ according to whether $n$ is even or odd.

\begin{lemma}
\label{lem:B-restriction}
    Assume that $n\ge 3$.
    \begin{enumerate}[(i)]
        \item If $n$ is even, $|\ca M_0(n)| \asymp |\ca M(n)| \asymp |\ca P^*_1(n)|$.
        \item If $n$ is odd, $|\ca N(n)| \asymp |\ca M(n)| \asymp |\ca P^*_{-1}(n)|$.
    \end{enumerate}
\end{lemma}

\begin{proof}
    Assume first that $n$ is even and $q$ is odd. The number of polynomials in $\ca P^*_1(n)$ with at least one root in $\{1, -1\}$ is $|\ca P^*_1(n)|(2/q - 1/q^2)$. It follows from \Cref{prop:parity_factors} that
    \[
    |\ca M(n)| \ge |\ca M_0(n)| = |\ca P^*_1(n)|(1 - (2/q - 1/q^2))/2.
    \]
    Since $|\ca P^*_1(n)|\geq |\ca M(n)|/2$, we get (i) when $q$ is odd.

   Assume now that $n$ is even and $q$ is even. The number of polynomials in $\ca P^*_1(n)$ with $1$ as a root is $|\ca P^*_1(n)|/q$. It follows from \Cref{prop:parity_factors} that
   \[
   |\ca M(n)| \ge |\ca M_0(n)| = |\ca P^*_1(n)|(1-1/q)/2.
   \]
   Since $|\ca P^*_1(n)|\geq \ca |\ca M(n)|$, we get (i) when $q$ is even.

   The proof of (ii) is similar.
\end{proof}

% \begin{lemma}
% 	\label{lem:B-restriction}
% Assume that $n\geq 3$.
% 	\begin{enumerate}[(i)]
% 		\item If $n$ is even and $q$ is even, $|\ca M_0(n)| / |\ca M(n)| \ge (1-1/q)/2$.
% 		\item If $n$ is even and $q$ is odd, $|\ca M_0(n)| / |\ca M(n)| \ge (1-2/q + 1/q^2)/4$.
% 		\item If $n$ is odd and $q$ is odd, $|\ca N(n)| / |\ca M(n)| \ge (1-2/q+ 1/q^2)/2$.
% 	\end{enumerate}
% \end{lemma}
% \begin{proof}
% (i) The number of polynomials in $\ca P^*_1(n)$ with $1$ as a root is $|\ca P^*_1(n)|/q$. It follows from \Cref{prop:parity_factors} that $|\ca M_0(n)| = |\ca P^*_1(n)|(1-1/q)/2$. Since $|\ca P^*_1(n)|\geq \ca |\ca M(n)|$, we get (i).

% (ii) The number of polynomials in  $\ca P^*_1(n)$ with at least one between $1$ and $-1$ as a root is  $|\ca P^*_1(n)|(2/q - 1/q^2)$. It follows from \Cref{prop:parity_factors} that $|\ca M_0(n)| = |\ca P^*_1(n)|(1 - (2/q - 1/q^2))/2 = |\ca P^*_1(n)|(1 - 2/q + 1/q^2)/2$. We have $|\ca P^*_1(n)|\geq |\ca M(n)|/2$, which gives (ii).

% (iii) Similar.
% \end{proof}
%\daniele[inline]{In fact, if needed, we can give exact formulas in the lemma, using the following facts: For $n$ even and $q$ even we have $|\ca P^*_1(n)|= |\ca M(n)| + |\ca M(n)| - |\ca P^*_1(n)|/q$. For $n$ even and $q$ odd we have $2|\ca P^*_1(n)|= |\ca M(n)| + |\ca P^*_1(n)|(2/q - 2/q^2) +3|P^*_1(n)| (1-2/q + 1/q^2)/2$.}

\section{Diagonally almost simple groups}
%\label{sec:main_proof}

In this section we will prove \Cref{thm:main-simple-case}. For the purpose of working up to the almost simple case it will be convenient to tackle the case of ``diagonally almost simple groups'', by which we mean almost simple groups of Lie type such that
\begin{equation}
    \label{eq:diag-almost-simple}
    S \le G \le \Inndiag(S),
\end{equation}
where $S$ is the socle of $G$.

\subsection{Bounded rank}

The bounded-rank case is comparatively easy (just as it is for uniform measure and classical Boston--Shalev).

\begin{proposition}
\label{prop:bounded-rank}
    Let $G$ be an almost simple group of Lie type satisfying \eqref{eq:diag-almost-simple}
    and acting faithfully and primitively on a set $\Omega$.
    Then $\derang(G,\Omega) \geq \eps(r)$ for a constant $\eps(r)>0$ depending on the rank $r$ of $G$.
\end{proposition}

\begin{proof}
    Note that $[G:S] \ll_r 1$.
    Applying \Cref{lem:passing-to-a-subgroup}, we may replace $G$ with $S$ and hence assume $G$ is a simple group of Lie type of rank $r$ acting transitively.
    Assume that $G$ has level $q$.
    Fulman--Guralnick~\cite{fulman2003derangements} showed that $\delta(G,\Omega)\geq \eps'(r)$ for a constant $\eps'(r)>0$.
    By \cite{guralnick_lubeck}, the proportion of regular semisimple elements in $G$ is $1-O(1/q)$.
    Hence the number of regular semisimple derangements in $G$ is at least $(\eps'(r) - O(1/q)) |G| \ge \eps''(r) |G|$, unless $G$ is among finitely many groups which we may ignore.
    For a regular semisimple element $g$ of $G$, $|g^G|\asymp_r |G|/q^r$, so
    we deduce that there are $\gg_r q^r$ conjugacy classes of $G$ consisting of regular semisimple derangements.
    Since $k(G) \asymp_r q^r$ (see \Cref{t:fulman_guralnick_bound}), the statement is proved.
\end{proof}

\subsection{Classical groups}

Having dispensed with bounded-rank groups, it now suffices to consider classical groups.
The bulk of the work in proving \Cref{thm:main-simple-case} consists of establishing the following variant for quasisimple classical groups.

\begin{proposition}
\label{thm:main-classical-quasisimple}
There exist positive absolute constants $\eps$ and $n_0$ such that the following holds.
Let $G$ be a classical group as in \eqref{eq:almost-quasisimple} with $n\geq n_0$, and let $M$ be a maximal subgroup of $G$ not containing $G'$.
Then $\derang(G,M) \ge \eps$.
\end{proposition}

We denote by $\derangss(G,\Omega)$ the proportion of semisimple conjugacy classes that consist of derangements (among all semisimple conjugacy classes).
By \Cref{t:fulman_guralnick_bound}, for a group as in \eqref{eq:almost-quasisimple} we have \[\derang(G,M) \gg \derangss(G,M).\]
Hence for \Cref{thm:main-classical-quasisimple} it is sufficient to show that $\derangss(G,M)\gg 1$.

\begin{remark}
\label{r:restriction}
In fact the restriction $n \ge n_0$ can be removed using \Cref{prop:bounded-rank}.
This is immediate if $Z=\Z(G)$ has bounded order, since in this case $k(G)\asymp k(G/Z)$. Hence we may assume that $G'=\SL_n^\pm(q)$. For bounded $n$, almost all classes of $G/Z$ are semisimple, and it is easy to see that the number of semisimple classes of $G/Z$ that lift to less than $|Z|$ classes of $G$ is $\ll (n,q\mp 1)q^{\floor{n/2}}$, and is $\ll 1$ if $G'=\SL_2(q)$. In particular, almost all classes of $G/Z$ lift to $|Z|$ classes of $G$, which implies that $\derang(G, M) \gg 1$.
\end{remark}

Let $(G, M)$ be as in \Cref{thm:main-classical-quasisimple}.
Let $Z = \Z(G)$.
Note that $Z \le M$, for otherwise by maximality we would have $G = M Z$ and hence $G' = M' \le M$.
Hence $M$ corresponds to a maximal subgroup $M / Z$ of the almost simple classical group $G / Z$ not containing its socle $(G/Z)' = G'Z/Z$, and hence Aschbacher's classification applies to $M$.

We point out at once that, for classes $\ca C_4, \ldots, \ca C_8, \ca S$, the bounds of Fulman--Guralnick \cite{fulman2012_conjugacy_classes} are already sufficient to prove \Cref{thm:main-classical-quasisimple}.

\begin{theorem}
\label{thm:main-quasisimple-classes-4+}
Let $G$ be a classical group as in \eqref{eq:almost-quasisimple}.
Let $r$ be the untwisted rank of $G$.
Let $M < G$ be a maximal subgroup not containing $G'$ in class $\ca C_i$ for some $i > 3$ or $\ca S$.
Then $\derang(G, M) = 1 - O(q^{-(r-1)/2})$.
\end{theorem}

\begin{proof}
\cite{fulman2012_conjugacy_classes}*{Lemma~7.6} asserts that $k(M) \ll q^{(r+1)/2}$. Given that the number of conjugacy classes of $G$ intersecting $M$ is at most $k(M)$, the statement follows from the fact that $k(G) \gg q^r$ (\Cref{t:fulman_guralnick_bound}).
\end{proof}

We now consider classes $\ca C_1, \ca C_2, \ca C_3$, which constitute the crux of the proof. Our analysis in these cases will use all the results of \Cref{sec:anatomy} on anatomy of polynomials. Recall the bijection $\Phi$ defined in \eqref{eq_linear}--\eqref{eq_orthogonal_odd}\phantomref{eq_unitary}\phantomref{eq_symplectic},
which will be used with no further explanation in the proofs of \Cref{t_c1,t_c2,t_c3}.

\subsection{Class $\ca C_1$}
\begin{theorem}
\label{t_c1}
Let $G$ be as in \eqref{eq:almost-quasisimple}, with $n$ large enough, and let $M$ be a maximal subgroup of $G$ of class $\ca C_1$. Then $\derangss(G, M) \geq \eps_1$ for an absolute constant $\eps_1 >0$.
\end{theorem}

\begin{proof}
In this proof, we will say that ``a conjugacy class $C$ fixes a $k$-space'' if and only if some (and therefore every) element of $C$ fixes a $k$-space.

\textbf{Case 1:} $\SL_n(q) \leq G \leq \GL_n(q)$. Then $M$ is the stabilizer of a $k$-space with $1\leq k \leq n-1$. A semisimple class $C$ fixes a $k$-space if and only if $\Phi(C)$ has a divisor of degree $k$, and equivalently if and only if $\Phi(C)$ has a divisor of degree $n-k$. Therefore the statement follows from Propositions \ref{prop:mult-problem}(1) (when $k$ is large) and \ref{prop:bounded_degree}(1) (when $k$ is bounded).

\textbf{Case 2:} $\SU_n(q) \leq G \leq \GU_n(q)$. Then $M$ is the stabilizer of a nondegenerate $k$-space for $1\leq k < n/2$, or the stabilizer of a totally singular $k$-space for $1\leq k \leq n/2$.
If a semisimple element of $G$ fixes a totally singular $k$-space, then it fixes a nondegenerate $2k$-space, and equivalently a nondegenerate $(n-2k)$-space; therefore we may assume that either $M$ is the stabilizer of a nondegenerate $k$-space for $1\leq k \le n/2$, or $n$ is even and $M$ is the stabilizer of a totally singular $n/2$-space.

Now note that a class $C$ fixes a nondegenerate $k$-space if and only if $\Phi(C)$ has a $\dagger$-symmetric divisor of degree $k$.
Similarly, $C$ fixes a totally singular $n/2$-space if and only if $\Phi(C)$ has the form $gg^\dagger$. Therefore the statement follows from Propositions \ref{prop:mult-problem}(3), \ref{prop:bounded_degree}(3), and \ref{prop:ff*}(3).

\textbf{Case 3:} $G=\Sp_{2n}(q)$. For $M\neq \SO^\pm_{2n}(q)$ (recall \Cref{convention_c1}), the proof is as in Case 2, using  Propositions \ref{prop:mult-problem}(2), \ref{prop:bounded_degree}(2) and \ref{prop:ff*}(2).

Therefore, assume $q$ is even and $M=\SO^\pm_{2n}(q)$. In accordance with \Cref{convention_c1}, identify $G$ with $\SO_{2n+1}(q)$, with module $V = \F_q^{2n+1}$, so that $M$ is the stabilizer of a nondegenerate hyperplane of $\pm$ type. Observe that if $g\in G$ does not have eigenvalue $1$ on the symplectic module $V / V^\perp$, then $g$ fixes only one nondegenerate hyperplane of the orthogonal module $V$, namely $[V,g]$. From the discussion preceding \Cref{fact} it follows that $g$ fixes a hyperplane of plus (respectively minus) type if and only if the number of irreducible factors of $\Phi(g^G)$ is even (respectively odd).
Therefore the statement follows from Proposition \ref{prop:parity_factors}.

\textbf{Case 4:} $G=\Omega=\Omega^\vareps_{n}(q)$ with $\vareps \in \{+,-,\circ\}$. Let $S=\SO^\varepsilon_{2n}(q)$ and $O=\Or^\varepsilon_{n}(q)$.
Rather than count semisimple conjugacy classes in $\Omega$ it suffices to count semisimple $O$-classes in $\Omega$, since $[O:\Omega] \le 4$.
By \Cref{prop:technical_4Z,l_equidistribution_cosets,lem:B-restriction}\phantomref{l_equidistribution_cosets}, it is equally sufficient to count semisimple $O$-classes in $S$.
Again by \Cref{lem:B-restriction}, it suffices to count classes $C$ in $\ca B = \ca B(n)$ (see \eqref{eq_orthogonal_even} and \eqref{eq_orthogonal_odd}).

% For $n$ even, the map $\Phi$, then, gives a bijection between this set and $\ca M_{(ii)}$, the set of polynomials of $\ca P^*_1(n)$ without $\pm 1$ as a root and with a prescribed parity of irreducible factors. For $n$ odd, the map $\Phi$ gives a bijection with the polynomials of the form $(X-1)f$ with $f\in \ca P^*_1(n-1)$ without $\pm 1$ as a root.

Assume first that $M$ is the stabilizer of a nondegenerate space of dimension $k \in [2, n/2]$ (of any type).
%If a semisimple element fixes a nondegenerate space of dimension $2j+1$ with $j\geq 1$, then it fixes a nondegenerate space of dimension $2j$. In particular, we may assume that $k$ is even.
Then $C$ fixes a nondegenerate $k$-space if and only if $\Phi(C)$ has a $*$-symmetric divisor of degree $k$.
If $k$ is large we conclude from Propositions~\ref{prop:parity_factors} and \ref{prop:mult-problem}(2).
If $k$ is bounded and $n$ is even we find from Proposition~\ref{prop:parity_factors} and \ref{prop:bounded_degree}(2) that a constant proportion of elements of $\ca M_0(n)$ have no divisor of degree at most $k$,
while if $k$ is bounded and $n$ is odd we find that a constant proportion of elements of $\ca N(n)$ have no divisor of degree at most $k$ other than $X-1$.

Assume now that $M$ is the stabilizer of a nondegenerate $1$-space, or that $q$ is even and $M$ is the stabilizer of a nonsingular vector.
If $C\cap M\neq \varnothing$, then $C$ has $\pm 1$ as eigenvalues. If $n$ is even this is impossible since $C\in \ca B$. If $n$ is odd, then $M$ is the stabilizer of a nondegenerate hyperplane of $\pm$ type. Since $C\in \ca B$, the elements of $C$ have an eigenvalue $1$ of multiplicity one and has no eigenvalue $-1$. In particular, the elements of $C$ fix only one nondegenerate hyperplane,
which is of plus (respectively minus) type if and only if the number of irreducible factors of degree at least two of $\Phi(C)$ is even (respectively odd). We conclude then by \Cref{prop:parity_factors}.

Assume finally that $M$ is the stabilizer of a totally singular space. As in Case 2, we need only to consider the case where it has dimension $n/2$. Then the characteristic polynomial decomposes as $gg^*$, and we conclude by \Cref{prop:mult-problem}(2).
\end{proof}

\begin{remark}
\label{r:tends_to_one}
It follows from the previous proof that, if $V$ is the natural module for $G$,  and $M$ is the stabilizer of a $k$-space with $k\leq (\dim V)/2$, and  $(G,M)\neq (\Sp_{2n}(q), \SO^\pm_{2n}(q))$, then $\derangss(G,\Omega) = 1 - O(k^{-\delta/3})$, where $\delta$ is as in \Cref{prop:mult-problem}.

This is straightforward in all cases, except when $M$ is the stabilizer of a totally singular $k$-space, where the argument needs to be slightly adjusted. Let us assume $\SU_n(q) \le G \le \GU_n(q)$; the other cases are analogous. A class $C$ of $G$ fixes a totally singular $k$-space if and only if $\Phi(C) = gg^\dagger h$, where $\deg(g) = k$. Let us bound the probability of this event. If $2k \le n/2$, then \Cref{prop:mult-problem}(3) gives a bound $\ll k^{-\delta}(1+\log k)^{-1/2}\ll k^{-\delta/3}$. If $n^{1/3} \le n-2k \le n/2$, then \Cref{prop:mult-problem}(3) gives a bound $\ll n^{-\delta/3} (\log n)^{-1/2} \ll k^{-\delta/3}$, since $\Phi(C)$ has a $\dagger$-symmetric divisor $h$ of degree $n-2k$. Finally, if $n-2k \le n^{1/3}$, then  \Cref{prop:ff*}(3) gives a bound $\ll n^{-1/6} \ll k^{-\delta/3}$.
\end{remark}

\subsection{Class $\ca C_2$}

\begin{theorem}
\label{t_c2}
Let $G$ be as in \eqref{eq:almost-quasisimple}, and let $M$ be a maximal subgroup of $G$ of class $\ca C_2$. Then, $\derangss(G,M) = 1 - O(n^{-\delta}(\log n)^{-1/2})$.
\end{theorem}

\begin{proof}
\textbf{Case 1:} $\SL_n(q) \leq G \leq \GL_n(q)$. Then $M$ is the stabilizer of a direct sum decomposition $\F_q^n=W_1\oplus \cdots \oplus W_t$, with $W_i \cong \F_q^m$ for every $i$, $m<n$ and $mt=n$. In particular, $M=G\cap L$ where $L = \GL_m(q) \wr S_t$. We will bound the number of conjugacy classes of $G$ that intersect $L$.

Set $N=\GL_m(q)^t$. Let $A \subset S_t$ be a set of conjugacy class representatives for the elements with at most $2t/3$ cycles.
Let $B \subset S_t$ be a set of conjugacy class representatives for the remaining elements.

By the description of the conjugacy classes in the wreath product $L$, the number of $N$-classes in $N\pi$ for any $\pi \in S_t$ is $k(\GL_m(q))^s$, where $s$ is the number of cycles of $\pi$. Using the bounds $k(\GL_m(q)) \le q^m$ (\cite{fulman2012_conjugacy_classes}*{Section 3.2}) and $|A| \le k(S_t) \le O(1)^{\sqrt t}$, it follows that the number of $\GL_n(q)$-classes intersecting $NA$ is at most
\[
    k(\GL_m(q))^{2t/3} |A| \le q^{2n/3} O(1)^{\sqrt t}.
\]
Multiplying by a factor of $q-1$, we get an upper bound for the number of $G$-classes intersecting $NA$ which is much smaller than the number of semisimple classes of $G$, which is at least $q^{n-1}$.

Next consider the $G$-classes intersecting $NB$.
Each element $b \in B$ has more than $2t/3$ cycles and therefore at least $t/3$ fixed points.
In particular each $b \in B$ fixes a set of size $s = \ceil{t/3}$.
Therefore if a $G$-class $C$ intersects $NB$, then each element of $C$ fixes a space of dimension $ms \ge n/3$. For a semisimple class, this happens with probability $\ll n^{-\delta}(\log n)^{-1/2}$ by \Cref{prop:mult-problem}(1). This concludes the proof in this case.

\textbf{Case 2:} $\SU_n(q) \leq G \leq \GU_n(q)$. Assume first $V=V_1\perp \cdots \perp V_t$ with $V_i$ nondegenerate of dimension $m$, so $n=mt$ and $M=L\cap G$ with $L=\GU_{m}(q)\wr S_t$. Then a similar proof as in Case 1 holds, using Proposition \ref{prop:mult-problem}(3); cf. \cite{fulman2018extension_field}*{Theorem~4.4}.

Assume then $V=V_1\oplus V_2$ with $V_i$ totally singular, so $M= L\cap G$ where $L=\GL_{n/2}(q^2) \rtimes C_2$. More precisely, letting $X=\GL_{n/2}(\overline{\F_q})$,  $\varphi_q$ be the morphism which raises each matrix entry to its $q$-th power, and $\gamma$ be the inverse-transpose map, we have $L\cong X_{\varphi_q^2}\rtimes \gen{\gamma\varphi_q}$.
% Here maybe I was a bit sketchy. Here are more details, which we may decide to partially include. Choosing bases $e_1, \ldots, e_{n/2}$ and $f_1, \ldots, f_{n/2}$ of $V_1$ and $V_2$, a matrix calculation shows that an element of the subgroup $\GL_{n/2}(q^2)$ has the form $y=(g, g^{-Tq})$ with respect to this basis (see e.g. \cite[Proposition 4.1.9]{kleidman_liebeck}), where $g\in \GL_{n/2}(q^2)$. Now the linear map $x$ swapping $e_i$ and $f_i$ for every $i$ is such that $y^x = y^{-Tq}$, which gives what we want. The symplectic case below is analogous.
In particular, by Lemma \ref{l_shintani_descent}, the number of $X_{\varphi_q^2}$-classes in the coset $X_{\varphi_q^2}\gamma\varphi_q$ is equal to the number of conjugacy classes of $X_{\gamma\varphi_q}=\GU_{n/2}(q)$, which is $\asymp q^{n/2}$. This is small compared to $k(G)$, since $k(G) \gg q^{n-1}$ by \Cref{t:fulman_guralnick_bound}.

If, on the other hand, a semisimple element $g$ fixes both $V_1$ and $V_2$, then $\Phi(g^G)=ff^\dagger$, and this happens with probability $\ll n^{-1/2}$ by Proposition \ref{prop:ff*}(3).

\textbf{Case 3:} $G=\Sp_{2n}(q)$. Assume first that $V=V_1\perp \cdots \perp V_t$ with $V_i$ nondegenerate of dimension $2m$, so $n=mt$ and $M=\Sp_{2m}(q)\wr S_t$. Then exactly the same argument given in Case 1 works, using Proposition \ref{prop:mult-problem}(2).

Assume then $V=V_1\oplus V_2$ with $V_i$ totally singular, so $M=\GL_n(q) \rtimes C_2$. We have $M\cong \GL_n(q) \rtimes \gen{x}$, where $x$ is the inverse-transpose map. By Lemma \ref{l_shintani_analogue}, the number of classes in the nontrivial coset is equal to the number of classes $C$ of $\GL_n(q)$ such that $C=C^{-1}$. The number of such classes is $\ll q^{n/2}$ by \cite{fulman2018extension_field}*{Lemma~4.9}.

On the other hand, if $g$ semisimple fixes both $V_1$ and $V_2$ then $\Phi(g^G) = ff^*$, which happens with probability $\ll n^{-1/2}$ by Proposition \ref{prop:ff*}(2).

\textbf{Case 4:} $G=\Omega=\Omega^\varepsilon_{n}(q)$ with $\varepsilon\in \{+,-,\circ\}$. Put $S=\SO^\varepsilon_{n}(q)$ and $O=\Or^\varepsilon_{n}(q)$. Given \Cref{prop:technical_4Z,l_equidistribution_cosets,lem:B-restriction}, it is sufficient to count semisimple $O$-classes in $S$. By \Cref{lem:B-restriction}, we may just count polynomials in $\ca P^*_{(-1)^n}(n)$, and the result for the $O$-classes in $S$ will follow.
Given this reduction, exactly the same argument given in Case 3 applies.
% Assume first that $V=V_1\perp \cdots \perp V_t$ with $V_i$ nondegenerate of dimension $m$, so $n=mt$, $M = G\cap L$ with $L=\Or^\xi_{m}(q)\wr S_t$, with $\xi^t = \varepsilon$ in the case $m$ even.
% Note that, given ??, it is sufficient to count among $A$-classes in $S$. In turn, given ??, it is sufficient to count among polynomials in $\ca P^*_1(n)$.
% Then exactly the same argument given in Case 1 works, using ?? and noting that, by ??, we may count among polynomials in $\ca P^*_1(n)$.
% Assume then $V=V_1\oplus V_2$ with $V_i$ totally singular, so $M=\GL_n(q): C_2$. The argument given in Case 3 works unchanged.
\end{proof}

\subsection{Class $\ca C_3$}

\begin{theorem}
\label{t_c3}
Let $G$ be as in \eqref{eq:almost-quasisimple}, and let $M$ be a maximal subgroup of $G$ of class $\ca C_3$. Then $\derangss(G,M) = 1 - O(n^{-1/4}\log n)$.
\end{theorem}

\begin{proof}
Put $K=\overline{\F_q}$. Throughout the proof, $\varphi_q$ denotes the endomorphism of $\GL_n(K)$ obtained by raising each matrix entry to its $q$-th power. Whenever we write $G\rtimes\gen{\varphi_q}$ where $G$ is a finite group, it is implicit that we mean the restriction of $\varphi_q$ to $G$.

\textbf{Case 1:} $\SL_n(q) \leq G \leq \GL_n(q)$. Assume that $M=G\cap L$, where $L=\GL_{n/r}(q^r) : r$ and $r$ is prime. Putting $X=\GL_{n/r}(K)$, we may write $L=X_{\varphi_q^r}\rtimes\gen{\varphi_q}$. Then, by Lemma \ref{l_shintani_descent}, the number of $X_{\varphi_q^r}$-classes in $L\sm X_{\varphi_q^r}$ is equal to $(r-1) k(X_{\varphi_q})\ll q^{n/2}$. Therefore, the number of $G$-classes intersecting $L\sm X_{\varphi_q^r}$ is $\ll q^{n/2+1}$, which is small.

Now, if a semisimple element of $G$ is contained in $X_{\varphi_q^r}$, then $\Phi(g^G)$ satisfies property $P_r$, as defined before Proposition \ref{prop:p_r}. In particular, by Proposition \ref{prop:p_r}(1) this happens with probability $\ll n^{-1/4}$, which concludes the proof in this case.

\textbf{Case 2:} $\SU_n(q) \leq G \leq \GU_n(q)$. Assume $M=L \cap G$, where $L=\GU_{n/r}(q^r) : r$ with $r$ odd prime. Putting $X=\GL_{n/r}(K)$ and letting $\gamma$ be the inverse-transpose map, we may write $L=X_{\gamma \varphi^r_q}\rtimes\gen{\varphi^2_q}$.
Since $r$ is odd, we have $L=X_{\gamma \varphi^r_q}\rtimes\gen{\gamma\varphi_q}$.
Thus by \Cref{l_shintani_descent}, the number of $X_{\gamma \varphi_q^r}$-classes in $L \sm X_{\gamma\varphi_q^r}$ is $(r-1) k(X_{\gamma \varphi_q})= (r-1) k(\GU_{n/r}(q))\ll q^{n/r}$, which is small.
 %As in Case 1, using Lemma \ref{l_shintani_descent}, we see that the number of $L$-classes in $L\sm N$ is $\ll q^{n/3}$. Therefore, the number of $G$-classes intersecting $L\sm N$ is $\ll q^{n/3+1}$, which is small.

In order to bound the number of semisimple $G$-classes intersecting $X_{\gamma\varphi_q^r} = \GU_{n/r}(q^r)$, we argue as in Case 1, using \Cref{prop:p_r}(3).

\textbf{Case 3:} $G=\Sp_{2n}(q)$. The cases to consider are $M=\Sp_{2n/r}(q^r):r$ with $2n/r$ even, and $M=\GU_{n}(q):2$. The first case is treated as in Case 1, using \Cref{l_shintani_descent} and \Cref{prop:p_r}(2).

Consider now the second case. Denoting by $X=\GL_n(K)$ and by $\gamma$ the inverse-transpose map, we may write $M=X_{\gamma\varphi_q}\rtimes \gen{\varphi_q}$. By \Cref{l_shintani_descent,l_shintani_analogue}, the number of $X_{\gamma\varphi_q}$-classes in the nontrivial coset is equal to the number of classes $C$ of $X_{\varphi_q}$ which are stable under $\gamma$, that is, such that $C=C^{-1}$. The number of such classes is $\ll q^{n/2}$ (see \cite{fulman2018extension_field}*{Lemma~4.9}). In order to bound the number of $G$-classes intersecting $\GU_n(q)$, use \Cref{prop:p_r}(2).

% The second case follows similarly (see \cite[proof of Theorem 5.11]{fulman2018extension_field}).

\textbf{Case 4:} $G=\Omega=\Omega^\varepsilon_n(q)$ with $\varepsilon \in \{+,-,\circ\}$. Put $S=\SO^\varepsilon_{n}(q)$ and $O=\Or^\varepsilon_{n}(q)$. By the same argument given in the proof of Case 4 in Theorem \ref{t_c2}, it is sufficient to count among polynomials in $\ca P^*_{(-1)^n}(n)$, up to paying an error $O(n^{-1/4}\log n)$ given by \Cref{prop:technical_4Z}.

Assume first $M=G\cap (\GU_{n/2}(q): 2)$. Then the same argument given in Case 3 applies.

Assume then $M=G \cap L$, with $L=\Or^{\varepsilon'}_{n/r}(q^r)\mathbin{.}r$ with either $n/r$ even and $\varepsilon' = \varepsilon$, or $n/r$ odd and $\varepsilon' = \circ$. Write $X=\SO_{n/r}(K)$ if $q$ is odd, and $X=\Omega_{n/r}(K)$ if $q$ is even. Let $\gamma\in \Or_{n/r}(K)$ be a reflection with $[\varphi_q,\gamma]=1$. We have $\Or^{\varepsilon'}_{n/r}(q^r) \cong X_{\gamma'\varphi^r_q} \rtimes{\gen \gamma}$, where $\gamma'=1$ if $\varepsilon' = +$ or $\circ$, and $\gamma'=\gamma$ if $\varepsilon' =-$. Moreover, $L=\gen{X_{\gamma'\varphi^r_q}, \gamma,\varphi_q}$. The coset $\Or^{\varepsilon'}_{n/r}(q^r)\varphi_q$ is the union of $X_{\gamma'\varphi_q^r}\gamma\varphi_q$ and $X_{\gamma'\varphi_q^r}\varphi_q$. By \Cref{l_shintani_descent} the number of $X_{\gamma'\varphi_q^r}$-classes in the two cosets is at most $k(X_{\gamma\varphi_q})$ and $k(X_{\varphi_q})$, respectively, which are $\ll q^{n/(2r)}$. % If $\varepsilon'=-$, then the coset $\Or^{\varepsilon'}_{n/r}(q^r)\varphi_q$ is the union of $X_{\gamma \varphi_q^r}\gamma\varphi_q$ and $X_{\gamma\varphi_q^r}\varphi_q$, and the same argument applies.
In order to bound the $G$-classes intersecting $\Or^{\varepsilon'}_{n/r}(q^r)$, we use \Cref{prop:p_r}(2).
\end{proof}

\subsection{Finishing the proof}

The proof of \Cref{thm:main-classical-quasisimple} now follows immediately.

\begin{proof}[Proof \Cref{thm:main-classical-quasisimple}]
Let $G$ be a group as in \eqref{eq:almost-quasisimple}, and let $M$ be a maximal subgroup of $G$ not containing $G'$. In particular, $Z=\Z(G)\leq M$, hence $M$ corresponds to a maximal subgroup of the almost simple group $G/Z$ not containing the socle $G'Z / Z$, so it belongs, in Aschbacher's description, to class $\ca C_i$ for some $i\in \{1, \dots, 8\}$ or $\ca S$.
We may assume that the rank of $G'$ is large.
Now the $i = 1, 2, 3$ cases follow from \Cref{t_c1,t_c2,t_c3}\phantomref{t_c1}\phantomref{t_c3} together with \Cref{t:fulman_guralnick_bound},
while the other cases follow from \Cref{thm:main-quasisimple-classes-4+}.
\end{proof}

Having now finished the proof of \Cref{thm:main-classical-quasisimple}, we can now deduce \Cref{t_main} in the diagonally almost simple case \eqref{eq:diag-almost-simple}.
In particular this completes the proof of \Cref{thm:main-simple-case}.

\begin{proposition}
    \label{prop:diag-almost-simple-case}
    \Cref{t_main} holds in the case \eqref{eq:diag-almost-simple}.
\end{proposition}
\begin{proof}
If the rank is bounded then we may apply \Cref{prop:bounded-rank}.
Hence we may assume $G$ is a classical almost simple group
of sufficiently large rank.

There exists a normal subgroup $S$ of $G$ with $\Soc(G) \le S \le G$ and $|G:S|\leq 4$ where $S$ is either simple or $\PSL_n^\pm(q) \leq S \leq \PGL_n^\pm(q)$.
Let $\Delta$ be a maximal system of imprimitivity for $S$.
Since $\Soc(G)$ is transitive on $\Omega$ and monolithic in $S$, it follows that $S$ acts primitively and faithfully on $\Delta$.
Hence, using \Cref{lem:passing-to-a-subgroup} and the fact that $\delta(S,\Omega)\ge \delta(S,\Delta)$, we may replace $(G, \Omega)$ with $(S, \Delta)$.
Thus we may assume $G$ is either simple or in an interval $\PSL_n^\pm(q) \le G \le \PGL_n^\pm(q)$.%\footnote{In fact $S$ acts primitively on $\Omega$: see \cite{kleidman_liebeck}*{Proposition~3.1.3}.}

In the linear and unitary cases define $\Gamma$ to be the preimage of $G$ in $\GL_n^\pm(q)$.
In the symplectic and orthogonal cases define $\Gamma = \Sp_{2n}(q)$ and $\Gamma = \Omega_n^\vareps(q)$, respectively.
Then $\Gamma$ is a group as in \eqref{eq:almost-quasisimple} such that $\Gamma / Z = G$, where $Z = \Z(\Gamma)$.
The action of $G$ lifts to an action of $\Gamma$ with point stabilizer equal to some maximal subgroup $M$ not containing $\Gamma'$, since $\Gamma'$ is the preimage of $\Soc(G)$.
By \Cref{thm:main-classical-quasisimple}, $\derang(\Gamma, \Omega) \gg 1$.

Since the obvious map from conjugacy classes of $\Gamma$ to conjugacy classes of $G$ is no worse than $|Z|$-to-$1$,
the number of conjugacy classes of $G$ containing derangements is at least $k(\Gamma) \derang(\Gamma, \Omega) / |Z|$.
To complete the proof it therefore suffices to observe that $k(G) \ll k(\Gamma) / |Z|$.
This is trivial if $|Z|$ is bounded, since $k(G) \le k(\Gamma)$, and in the linear and unitary cases it follows from \Cref{lem:k-for-GL-and-PGL} since $|Z| = q \mp 1$.
\end{proof}

We can also prove \Cref{t:semisimple_asymptotic}.

\begin{proof}[Proof of \Cref{t:semisimple_asymptotic}]
    Let $G$ be a finite simple group of rank $r$ and level $q$, and assume that $\derangss(G,\Omega)\leq 1-\eps$. Assume first that $G$ is classical. We first work with a group $L$ as in \eqref{eq:almost-quasisimple} with $G=L/\Z(L)$, and we assume $\derangss(L,\Omega)\leq 1-\eps$; at the end of the proof, we will indicate how to deduce the result for $G$. Assume we are not in case (ii) of the statement. By \Cref{t_c2,t_c3,thm:main-quasisimple-classes-4+}, there exists $f_2(\eps)$ such that if $r\geq f_2(\eps)$ then $L_\alpha$ is of class $\ca C_1$. Assume that this is the case, so $L_\alpha$ is the stabilizer of a $k$-subspace of the natural module $V$. By \Cref{t_c1,r:tends_to_one}\phantomref{r:tends_to_one}, there exists $f_1(\eps)$ such that either $k$ or $\dim V - k$ is at most $f_1(\eps)$. This is case (i) in the statement. Assume now that $r \le f_2(\eps)$. We now include the case where $G=L$ is exceptional. By \Cref{prop:bounded-rank} and by \cite{fulman2003derangements}, there exists $f_3(\eps)>0$ such that either $|L| \le f_3(\eps)$ or $L_\alpha$ is of maximal rank, so either (iii) or (iv) holds.

    This concludes the proof, except that for classical groups we have worked with $L$ rather than with $G=L/Z$, where $Z=\Z(L)$. Assuming that $\derangss(G,\Omega)\le 1-\eps$, it follows that $\derangss(L,\Omega)\le 1-\eps/|Z|$. In particular, if $|Z|$ is bounded the result for $G$ follows from the result for $L$, proved in the previous paragraph. The remaining case is $G=\PSL^\pm_{n}(q)$. As observed in \Cref{r:restriction}, almost all semisimple classes of $G$ lift to $|Z|$ classes of $L$, which implies that $\derangss(G,\Omega)=\derangss(L,\Omega)+o(1)$, so again the result for $G$ follows from the result for $L$.
\end{proof}

\section{Almost simple groups}

The main extra ingredient in the proof of \Cref{t_main} is a bound on the number of conjugacy classes of an almost simple group of Lie type, which is of independent interest (see \Cref{thm:conj_classes_almost_simple} below).
%Loosely speaking, the result asserts that adding field automorphisms to a simple group of Lie type decreases the number of conjugacy classes (see \Cref{thm:conj_classes_almost_simple} below).

We need an auxiliary result, and we begin by recalling some standard material about algebraic groups (refer to \cite{malle_testerman}).
Let $X$ be a simple linear algebraic group over $\overline{\F_p}$ which is either simply connected or adjoint.
Let $r$ denote the rank of $X$. Fix a maximal torus $T$ and a Borel subgroup $B$ containing $T$. Let $X(T) = \Hom(T, \mathbf G_m)$ be the character group of $T$, and let $W = \mathrm N_X(T)/T$ be the Weyl group of $X$ with respect to $T$. We can choose a scalar product $(\cdot, \cdot)$ on $X(T)\otimes_\Zint \R$ preserved by $W$.
%Let $\Phi$ be a set indexing the root subgroups $U_\alpha$~($\alpha \in \Phi$) of $T$. For every $\alpha \in \Phi$, fix an isomorphism $u_\alpha \colon \mathbf G_a \to U_\alpha$, so we get a homomorphism  (called root) $T \to \mathbf G_m$, which we denote by $\alpha$ and which is given by $T\to \Aut(U_\alpha) \cong \Aut(\mathbf G_a) \cong \mathbf G_m$. Accordingly, we view $\Phi$ as a subset of $X(T)$. It follows that
Let $\Phi$ be the root system of $X$ with respect to $T$, and let $\Delta\subset \Phi$ be the base with respect to $(T,B)$. Let $U_\alpha$~($\alpha \in \Phi$) be the corresponding root subgroups. Fix isomorphisms $u_\alpha \colon \mathbf G_a \to U_\alpha$ for every $\alpha \in \Phi$ such that
\[
    (x^{u_\alpha})^t = (xt^\alpha)^{u_\alpha} \qquad (x \in \mathbf G_a, \alpha \in \Phi, t\in T).
\]
%(see \cite{malle_testerman}*{Theorem~8.17}).
% Let $\phi$ be an endomorphism of $X$ stabilizing the pair $(T,B)$. Then $\phi$ acts on $X(T)$ via the formula
% \[
% 	t^{\chi \phi} = (t^\phi)^\chi \qquad (t \in T, \chi \in X(T)).
% \]
% Moreover, by definition $\phi$ permutes the root subgroups, and accordingly it induces a permutation $\tau_\phi$ of the roots via the formula $U_{\alpha \tau_\phi} = U_\alpha^\phi$. This action preserves $\Delta$. In the case where $\phi$ is an automorphism of $X$ (as an algebraic group), we have the following equation in $X(T)$ (see \cite{malle_testerman}*{Lemma~11.10}):
% \[
% 	\alpha\phi = \alpha \tau_\phi^{-1} \qquad (\alpha \in \Phi),
% \]
% which is to say
% \begin{equation}
% \label{eq:same}
% 	(t^{\phi})^\alpha = t^{\alpha \tau^{-1}_\phi} \qquad (t \in T, \alpha \in \Phi).
% \end{equation}
We will use the notation
\begin{equation}
\label{eq:endomorphisms}
\varphi, \gamma_2, \gamma_3, \rho
\end{equation}
for certain endomorphisms of $X$. First, $\varphi$ denotes the Frobenius endomorphism $x\mapsto x^p$ of $X$ stabilizing $(T,B)$ (see \cite{malle_testerman}*{Theorem~16.5} or \cite{gorenstein}*{Theorem 1.15.4(a)}). More precisely, $\varphi$ acts on the root subgroups according to
\[
	(x^{u_\alpha})^\varphi = (x^p)^{u_\alpha} \qquad (x \in \mathbf G_a, \alpha \in \Phi).
\]
and the torus according to
\[
	(t^{\varphi})^\alpha = (t^\alpha)^p \qquad (t \in T, \alpha \in \Phi).
\]
Next, consider a permutation $\tau$ of $\Delta$ corresponding to a symmetry of the Dynkin diagram and extending to an isometry of $X(T) \otimes_\Zint \R$.
The \emph{graph automorphism} $\gamma_\tau$ is the automorphism of $X$ commuting with $\varphi_p$, stabilizing $(T,B)$, and acting on the root subgroups according to
\[
(x^{u_\alpha})^{\gamma_\tau} = x^{u_{\alpha \tau}} \qquad (x \in \mathbf G_a, \alpha \in \Phi)
\]
(see \cite{gorenstein}*{Theorem 1.15.2}).
In particular, $|\gamma_\tau| = |\tau|$ and $\gamma_\tau$ induces the permutation $\tau$ on the set of root subgroups.
If $X$ is of type $A_{r+1}~(r\geq 2)$, $D_r~(r\geq 4)$, or $E_6$ we let $\gamma_2=\gamma_\tau$ where $|\tau|=2$.
If $X$ is of type $D_4$ we similarly let $\gamma_3=\gamma_{\tau}$ where $|\tau|=3$.
In all other cases, we define $\gamma_2=\gamma_3=1$.
Note that if $X = D_4$ then $\gen{\gamma_2, \gamma_3}\cong S_3$.

Finally, when $X=B_2$, $F_4$, or $G_2$, with $p=2$, $2$, or $3$ respectively, $\rho$ is the \emph{graph-field endomorphism} of $X$ as in \cite{gorenstein}*{Theorem~1.15.4(b)}.
In particular, $\rho^2=\varphi$.
In all other cases define $\rho = 1$.

\begin{lemma}
\label{l:stable}
Let $X=A_{r+1}~(r\geq 2)$ or $D_r~(r\geq 4)$ or $E_r~(r=6)$ be simply connected, and let $1 \ne \gamma \in \{\gamma_2, \gamma_3\}$. Let $q=p^f$ and $\sigma=\varphi^f$, so $X_\sigma$ has level $q$. Then the number of $\gamma$-stable semisimple conjugacy classes of $X_\sigma$ is $ q^\ell$, where $\ell$ is the number of orbits of $\gamma$ on $\Delta$.
\end{lemma}

\begin{proof}
Let $\Delta =\{\alpha_1, \ldots, \alpha_r\}$. Let $\{\lambda_1, \ldots, \lambda_r\}$ be the set of fundamental dominant weights with respect to $\Delta$; i.e., $(\lambda_i, \alpha^*_j) = \delta_{ij}$, where $\alpha^*_j = 2\alpha_j/(\alpha_j, \alpha_j)$. Since $\gamma$ induces an isometry of $X(T) \otimes_\Zint \R$, it acts in the same way on $\Delta$ and on $\{\lambda_1, \ldots, \lambda_r\}$.

Equivalence classes of irreducible representations of $X_\sigma$ over $\overline{\F_p}$ are parametrized by tuples $(c_1, \ldots, c_r)$, where $0\leq c_i \leq q-1$ for every $i$ (see \cite{kleidman_liebeck}*{Theorem 5.4.1}). It follows from the previous paragraph and from \cite{kleidman_liebeck}*{Proposition 5.4.2(ii)}  that the action of $\gamma$ on the set of representations corresponds to the action on tuples where $\gamma$ permutes the coordinates as the elements of $\Delta$. In particular, the number of fixed points of $\gamma$ in this action is exactly $q^\ell$.

Consider the complex vector space $V$ of functions $X_\sigma \to \mathbf C$ supported on $p'$-elements (i.e., semisimple elements). Both the indicator functions on classes of $p'$-elements and the irreducible Brauer characters of $X_\sigma$ are bases for $V$ (see \cite{navarro1998characters}*{Corollary 2.10}); call them $B_1$ and $B_2$, respectively.
Now $\gamma$ acts on $V$ and preserves both bases $B_1, B_2$.
By invariance of trace it follows that $\gamma$ fixes the same number of points in $B_1$ as in $B_2$.
The number of fixed points on $B_2$ is $q^\ell$, while the number of fixed points on $B_1$ is the number of $\gamma$-stable semisimple conjugacy classes of $X_\sigma$. This finishes the proof.
\end{proof}

\begin{lemma}
Let $S = (X_\sigma)'$ be a finite simple group of Lie type of untwisted rank $r$ and level $q$.
Let $I = \Inndiag(S) = X_\sigma$ and $A = \Aut(S)$.
The number of $I$-classes in $A \sm I$ is bounded by
\[
    O(q^{r-1+a} / |I:S|),
\]
where $a = 1/2$ for $\PSL_2(q)$ and $a = 0$ in all other cases.
\end{lemma}
\begin{proof}
Write $q^u=p^f$, where $u \in \{1, 2, 3\}$ is the order of the permutation $\tau_\sigma$ induced by $\sigma$ on the Dynkin diagram.
We call $X_\sigma$ \emph{twisted} if $u > 1$.
Fix a $\sigma$-stable maximal torus $T$ and a $\sigma$-stable Borel subgroup $B$ containing $T$, and let $\varphi, \gamma_2, \gamma_3, \rho$ be as in \Cref{eq:endomorphisms} corresponding to this choice. We may assume that one of the following holds (see \cite{gorenstein}*{Theorem~2.2.3}):
\begin{itemize}[$\diamond$]
    \item $X_\sigma$ is untwisted and $\sigma=\varphi^f$,
    \item $X_\sigma$ is twisted but not Suzuki or Ree and $\sigma =\gamma_u\varphi^{f/u}$, or
    \item $X_\sigma$ is Suzuki or Ree and $\sigma=\rho^f$.
\end{itemize}
In all cases the quotient $A / I$ is an abelian group generated by the images of $\gamma_2$, $\gamma_3$, $\varphi$, and $\rho$
(see \cite{gorenstein}*{Theorem~2.5.12}).

The main tools used in this proof are \Cref{l_shintani_analogue,l_shintani_descent}.
Note particularly that \Cref{l_shintani_analogue} implies that the number of $I$-classes in a coset $Ix$ is the same as that in $Iy$ if $\gen x = \gen y$.
More generally, the number of $I$-classes in $Ix$ is at most that in $Iy$ whenever $y \in \gen x$.
Therefore it suffices to bound the number of $I$-classes in
\begin{enumerate}[(i)]
    \item $I \gamma_v \varphi^{f/v}$ $(v \in \{2, 3\})$, $\gamma_v \ne 1$, $X_\sigma$ untwisted,
    \item $I \gamma_v$ ($v \in \{2, 3\}$), $\gamma_v \ne 1$, $X_\sigma$ untwisted,
    \item $I \gen {\varphi} \sm I$ where $X_\sigma$ is not ${}^uB_2(2^f)$, ${}^uF_4(2^f)$, or ${}^uG_2(3^f)$ for any $u\in\{1,2\}$.
    \item $I \gen {\rho} \sm I$ where $X_\sigma$ is ${}^uB_2(2^f)$, ${}^uF_4(2^f)$, or ${}^uG_2(3^f)$ for some  $u\in \{1,2\}$.
\end{enumerate}

\textbf{Case (i):}
Let $\tau = \gamma_v \varphi^{f/v}$.
Applying \Cref{l_shintani_descent}, the number of $X_\sigma$-classes in $X_\sigma \tau$ is equal to the number of $X_\tau$-classes in $X_\tau \sigma$.
Since $X_\tau$ has rank $r$ and level $p^{f/v} = q^{1/v}$,
we have $k(X_\tau) \asymp q^{r/v} \ll q^{r-1} / r$ since $r, v \ge 2$.

\textbf{Case (ii):}
Here $X_\sigma$ is one of $A_r(q)$ ($r \ge 2)$, $D_r(q)$ ($r \ge 4$), $E_6(q)$.

Assume first that $X$ has type $A_r$ ($r \ge 2$) and let $n = r+1$, so $I = \PGL_n(q)$.
Taking $T$ and $B$ to the standard upper-triangular choices, we have that $\gamma = \gamma_2$ is the image of the inverse-transpose map $\tilde \gamma \in \Aut(\GL_n(q))$ composed with an inner automorphism.
Therefore it suffices to bound the number of $\GL_n(q)$-classes in $\GL_n(q)\tilde \gamma$. By \cite{fulman2018extension_field}*{Lemma~4.9}, this number is $\ll q^{\floor{n/2}} \ll q^{n-2}/(n-1) = q^{r-1} / r$, as required.

Consider now all the other cases; in particular, $|I:S|\leq 4$. Let $L$ be the group of simply connected type projecting to $S$, and let $\tilde{\gamma_v}$ be a lift of $\gamma_v$ to $\Aut(L)$. We first bound the number of classes in $L\tilde{\gamma_v}$. By \Cref{l_shintani_analogue}, their number is the same as the number of $\tilde{\gamma_v}$-stable classes in $L$. By \Cref{l:stable}, the number of semisimple such classes is $q^{\ell}$, where $\ell$ is the number of orbits of $\gamma_v$ on $\Delta$. The number of non-semisimple such classes is $O(q^{r-1})$, hence overall we get a bound $O(q^{r-1})$. This gives an upper bound for the number of classes in $S\gamma_v$. Now, consider any coset $Sg\gamma_v$ where $g\in I$. Using again \Cref{l_shintani_analogue} and the fact that there are $\ll q^{r-1}$ non-semisimple classes, we are reduced to count the number of $g\gamma_v$-stable semisimple classes in $S$. Any semisimple $S$-class is $I$-stable (see \cite{seitz1982generation}*{(2.12)}), so it is $g\gamma_v$-stable if and only if it is $\gamma_v$-stable, so we are reduced to the coset $S\gamma_v$, which we already considered. Given that there $|I:S| \ll 1$ possibilities for the coset $Sg\gamma_v$, we are done.

\textbf{Case (iii):}
Let $x = \varphi|_S$ and note that $|x| = f$.
By \Cref{l_shintani_analogue}, the number of $X_\sigma$-classes in $X_\sigma x^i$ is the same as the number of $X_\sigma$-classes in $X_\sigma x^d$ where $d = \gcd(i, f)$.
By \Cref{l_shintani_descent}, the number of $X_\sigma$-classes in $X_{\sigma}x^d$ is the same as the number of $X_{\varphi^d}$-classes in $X_{\varphi^d}\sigma$. There are at most $k(X_{\varphi^d})$ of these by \Cref{l_shintani_analogue}. Since the level of $X_{\varphi^d}$ is equal to $p^d$, by \Cref{t:fulman_guralnick_bound} we have $k(X_{\varphi^d}) \asymp p^{d r}$.
Thus the contribution from all $i$ with $\gcd(i, f) = d = f/e$ is bounded by
\[
    \phi(e) p^{dr} = \phi(e) q^{ur / e}.
\]
Summing over $e$ gives a bound of the form $O(q^{r-1+a} / r)$ provided that we exclude the cases
\begin{equation}
    \label{eq:rue}
    (r, u, e) =
    (r, 2, 2), (2, 2, 3),
    (4, 3, 2), (4, 3, 3).
\end{equation}
We have to deal with these cases separately.

Consider one of the cases in \eqref{eq:rue}.
In particular $X_\sigma$ is twisted, so $\sigma = \gamma_u \varphi^{f/u}$.
By the argument above we need to count $X_{\varphi^d}$-classes in $X_{\varphi^d} \gamma_u \varphi^{f/u}$.
This situation was considered in cases (i) and (ii) above, where we gave a bound $\ll q^{r-1}/r$. Here the level is $p^d$, so the number of classes is bounded by $\ll p^{d(r-1)} / r = q^{u(r-1)/e} / r$.
This leaves only the case $(r, u, e) = (4, 3, 2)$.
In this case we are counting $X_{\varphi^{f/2}}$-classes in $X_{\varphi^{f/2}} \tau$, where $\tau = \gamma_3 \varphi^{f/3}$ and $6 \mid f$.
Note that $X_{\varphi^{f/2}}$ has rank $4$ and level $q^{3/2}$, and that $\varphi^{f/3}$ is nontrivial on $X_{\varphi^{f/2}}$.
In this case we saw in (i) above that actually the number of $X_{\varphi^{f/2}}$-classes in $X_{\varphi^{f/2}} \tau$ is $\ll (q^{3/2})^{4/3} = q^2$, as required.

\textbf{Case (iv):}
Let $x = \rho|_S$ and note that $|\rho| = 2f/u$.
Arguing as in case (iii), the number of $X_\sigma$-classes in $X_\sigma x^i$ is equal to the number of $X_\sigma$-classes in $X_\sigma x^d$ where $d = \gcd(i, 2f / u)$, which is the same as the number of $X_{\rho^d}$-classes in $X_{\rho^d} \sigma$, which is at most $k(X_{\rho^d})$.
The level of $X_{\rho^d}$ is $p^{d/2}$, so $k(X_{\rho^d}) \asymp p^{dr/2}$ by \Cref{t:fulman_guralnick_bound}.
Thus the contribution from all $i$ with $\gcd(i, 2f/u) = d = 2f/ue$ is bounded by
\[
    \phi(e) p^{dr/2} = \phi(e) q^{r/e},
\]
which is an acceptable bound since $r, e \ge 2$ and $|I:S| \ll 1$.
This concludes the proof in all cases.
\end{proof}

\begin{remark}
    The previous lemma is a slight strengthening of \cite{fulman2012_conjugacy_classes}*{Lemma~5.4}.
    Since it is central to our result we have given more complete details.
\end{remark}

\begin{theorem}
\label{thm:conj_classes_almost_simple}
Let $G$ be an almost simple group of Lie type with socle $S$, untwisted rank $r$, and level $q$.
Let $I = \Inndiag(S)$ and $N = G \cap I$. Then
\[
k(G) = \frac{k(N)}{|G:N|} + O(q^{r-1+a}),
\]
where $a=1/2$ if $S \cong \PSL_2(q)$ and $a=0$ otherwise.
\end{theorem}

\begin{proof}
Since $|I:N| \le |I:S|$, the previous result implies that the number of $N$-classes in $G \sm N$ is $O(q^{r-1+a})$.
Applying \Cref{l_shintani_analogue}, we get the same bound for the number of conjugacy classes of $N$ that are stable under any $a \in G \sm N$.
Observe that $G/N$ acts semiregularly on the remaining classes, so permutes them in orbits of size $|G:N|$. This proves the theorem.
\end{proof}

Now we can finish the proof of \Cref{t_main}.
Assume now that $G$ is any almost simple group of Lie type with socle $S$. Let $N=G\cap \Inndiag(S)$. \Cref{thm:conj_classes_almost_simple} gives that $k(G)\asymp k(N)/{|G:N|}$, which implies $\derang(G,\Omega)\gg\derang(N,\Omega)$,
because each $G$-conjugacy class contained in $N$ splits into at most $|G:N|$ $N$-conjugacy classes.  Since $G$ is primitive and $N$ is normal in $G$, we have that $N$ is transitive. Let $\Delta$ be a maximal system of imprimitivity for $N$. Since $S$ is transitive on $\Omega$, $N$ acts primitively and faithfully on $\Delta$. We have $\derang(G, \Omega) \gg \derang(N, \Omega) \geq \derang(N,\Delta)\gg 1$ by \Cref{prop:diag-almost-simple-case}, which concludes the proof.

\providecommand*{\Lpolish}[0]{}
\bibliography{references}

\end{document}